\documentclass[10pt,reqno,a4paper]{amsart}

\usepackage{amsmath,amsfonts,amscd,amssymb,mathrsfs,eufrak,upgreek,enumitem,mathrsfs,colonequals}
\usepackage[usenames]{color}
\usepackage[mathscr]{euscript}
\usepackage{setspace}
\usepackage{enumitem}
\usepackage[font=small]{caption}
\usepackage[colorlinks]{hyperref}
\usepackage{graphicx}
\usepackage{stmaryrd}

\setlist[enumerate]{format=\normalfont}

\newcommand{\marginparstretch}{0.6}
\let\oldmarginpar\marginpar
\renewcommand\marginpar[1]{\-\oldmarginpar[\framebox{\setstretch{\marginparstretch}\begin{minipage}{\marginparwidth}{\raggedleft\tiny #1}\end{minipage}}]{\framebox{\setstretch{\marginparstretch}\begin{minipage}{\marginparwidth}{\raggedright\tiny #1}\end{minipage}}}}

\usepackage{tikz}
\usetikzlibrary{arrows,decorations.pathmorphing,decorations.pathreplacing,positioning,shapes.geometric,shapes.misc,decorations.markings,decorations.fractals,calc,patterns}
\tikzset{
        cvertex/.style={circle,draw=black,inner sep=1pt,outer sep=3pt},
        vertex/.style={circle,fill=black,inner sep=1.2pt,outer sep=4pt},
        DBs/.style={circle,draw=black,circle,fill=black,inner sep=0pt, minimum size=3pt},
        DB/.style={circle,draw=black,circle,fill=black,inner sep=0pt, minimum size=4pt},
         DWs/.style={circle,draw=black,circle,fill=white,inner sep=0pt, minimum size=3pt},
         DWds/.style={circle,draw=black,densely dotted,circle,fill=white,inner sep=0pt, minimum size=3pt},
        DW/.style={circle,draw=black,inner sep=0pt, minimum size=4pt},
        tvertex/.style={inner sep=1pt,font=\scriptsize},
        gap/.style={inner sep=0.5pt,fill=white},
        mid/.style={inner sep=0.5pt},
        Ggap/.style={inner sep=0.5pt,fill=green!40!black!20}}

\newtheorem{thm}{Theorem}[section]
\newtheorem{proposition}[thm]{Proposition}
\newtheorem{lemma}[thm]{Lemma}
\newtheorem{definition}[thm]{Definition}
\newtheorem{cor}[thm]{Corollary}
\newtheorem{conj}[thm]{Conjecture}

\theoremstyle{definition} 

\newtheorem{example}[thm]{Example}
\newtheorem{setup}[thm]{Setup}

\newtheorem{remark}[thm]{Remark}

\newtheorem{notation}[thm]{Notation}

\numberwithin{equation}{section}

\def\mod{\mathop{\rm mod}\nolimits}
\def\silt{\mathop{\sf silt}\nolimits}
\def\presilt{\mathop{\text{$2$-}\sf presilt}\nolimits}
\def\tilt{\mathop{\sf tilt}\nolimits}
\def\rig{\mathop{\sf rig}\nolimits}
\def\mrig{\mathop{\sf mrig}\nolimits}
\def\umrig{\mathop{\overline{\sf mrig}}\nolimits}
\def\twosilt{\mathop{\text{$2$-}\sf silt}\nolimits}
\def\twotilt{\mathop{\text{$2$-}\sf tilt}\nolimits}

\def\lsilt{\mathop{\text{$\ell$-}\sf silt}\nolimits}
\def\mut{\mathop{\sf mut}\nolimits}
\def\umut{\mathop{\overline{\sf mut}}\nolimits}

\def\proj{\mathop{\rm proj}\nolimits}

\def\uHom{\mathop{\underline{\rm Hom}}\nolimits}
\def\Hom{\mathop{\rm Hom}\nolimits}
\def\RHom{\mathop{\mathbf{R}\rm Hom}\nolimits}

\def\End{\mathop{\rm End}\nolimits}
\def\uEnd{\mathop{\underline{\rm End}}\nolimits}
\def\Ext{\mathop{\rm Ext}\nolimits}

\def\add{\mathop{\rm add}\nolimits}

\def\Spec{\mathop{\rm Spec}\nolimits}

\def\Db{\mathop{\rm{D}^b}\nolimits}
\def\Kb{\mathop{\rm{K}^b}\nolimits}

\newcommand{\con}{\mathrm{con}}

\def\tilt{\mathop{\sf tilt}\nolimits}
\def\Rf{{\rm\bf R}f}

\newcommand{\tT}{\EuScript{T}}
\newcommand{\rR}{\mathfrak{R}}
\newcommand{\sS}{\mathfrak{S}}
\newcommand{\scrE}{\EuScript{E}}
\newcommand{\cC}{\EuScript{C}}
\newcommand{\cD}{\EuScript{D}}
\newcommand{\cS}{\EuScript{S}}
\newcommand{\ucE}{\underline{\scrE}}
\newcommand{\id}{\mathrm{id}}

\DeclareMathOperator{\cm}{\mathrm{CM}}
\DeclareMathOperator{\ucm}{\underline{\mathrm{CM}}}
\newcommand{\cmr}{\cm R}
\newcommand{\cmrr}{\cm \rR}
\newcommand{\ucmrr}{\ucm \rR}

\newcommand{\ucmss}{\ucm \sS}
\newcommand{\cE}{\scrE}

\makeatletter
\def\mapstofill@{%
   \arrowfill@{\mapstochar\relbar}\relbar\rightarrow}
\newcommand*\xmapsto[2][]{%
   \ext@arrow 0395\mapstofill@{#1}{#2}}
\makeatother

\begin{document}
\title[Finiteness of Derived Equivalence Classes]{\textsc{On the Finiteness of the Derived Equivalence Classes of some Stable Endomorphism Rings}}
\author{Jenny August}
\address{Max Planck Institute for Mathematics,
Vivatsgasse 7,
53111 Bonn,
Germany.}
\email{jennyaugust@mpim-bonn.mpg.de}
\begin{abstract}

We prove that the stable endomorphism rings of rigid objects in a suitable Frobenius category have only finitely many basic algebras in their derived equivalence class and that these are precisely the stable endomorphism rings of objects obtained by iterated mutation. The main application is to the Homological Minimal Model Programme. For a 3-fold flopping contraction $f \colon X \to \Spec R$, where $X$ has only Gorenstein terminal singularities, there is an associated finite dimensional algebra $A_{\con}$ known as the contraction algebra. As a corollary of our main result, there are only finitely many basic algebras in the derived equivalence class of $A_\con$ and these are precisely the contraction algebras of maps obtained by a sequence of iterated flops from $f$. This provides evidence towards a key conjecture in the area.
\end{abstract}
\maketitle
\parindent 15pt
\parskip 0pt

\section{Introduction}

This paper focuses on a fundamental problem in homological algebra: given a basic algebra $A$, find all the basic algebras $B$ such that $A$ and $B$ have equivalent derived categories. We will give a complete answer to this question for a class of finite dimensional algebras arising from suitable Frobenius categories.

By a well known result of Rickard \cite{morita}, the above problem is equivalent to first finding all the \emph{tilting complexes} over $A$ and then computing their endomorphism rings. One approach to the first of these problems is to use \textit{mutation}; an iterative procedure which produces new tilting complexes from old. The naive hope is that starting from a given tilting complex, all others can be reached using mutation. However, for a general algebra, tilting complexes do not behave well enough for this to work. There are two key problems:
\begin{enumerate}
\item The mutation procedure does not always produce a tilting complex. 
\item It is often possible to find two tilting complexes which are not connected by any sequence of mutations.
\end{enumerate}
Both problems have motivated results in the literature; the first prompting the introduction of the weaker notion of \textit{silting complexes} \cite{kellervos, siltingmutation} and the second resulting in restricting to a class of algebras known as \textit{tilting-discrete} algebras \cite{tiltingconnectedness, AiharaMiz}. In this paper, we will combine these ideas to provide a class of examples of finite dimensional algebras for which the derived equivalence class can be completely determined.  
\subsection{Algebraic Setting and Results}
From an algebraic perspective, the algebras we consider arise in cluster-tilting theory, where the objects of study are \textit{rigid objects} in some category $\cC$ and the algebras of interest are their endomorphism algebras.
More precisely, let $\cE$ be a Frobenius category such that its stable category $\cC =\ucE$ is a $k$-linear, Hom-finite, Krull-Schmidt, 2-Calabi-Yau triangulated category with shift functor denoted $\Sigma$. This ensures $\ucE$ has the conditions usually imposed for cluster-tilting theory, but we shall add two additional assumptions: 
\begin{enumerate}
\item $\ucE$ has at least one but a finite number of maximal rigid objects, and;
\item $\Sigma^2 \cong \id$. 
\end{enumerate}
Utilizing the strong links between rigid objects and the silting theory of their endomorphism algebras developed in \cite{[AIR]}  leads to the following, which is our main result.
\begin{thm}[Corollary \ref{mainresult}] \label{mainresultintro}
With the conditions above, for any rigid object $M$ of $\ucE$, the basic algebras derived equivalent to $\uEnd_\cE(M)$ are precisely the stable endomorphism algebras of rigid objects obtained from $M$ by iterated mutation. In particular, there are only finitely many such algebras.
\end{thm}
For \textit{maximal rigid objects}, the conditions on $\ucE$ above ensure that any two such objects are connected by a sequence of mutations and so we obtain the following corollary of Theorem \ref{mainresultintro}.
\begin{cor}[Corollary \ref{maxrig}] \label{maxrigintro}
With assumptions as in \ref{mainresultintro}, for any maximal rigid object $M$ of $\ucE$, the basic algebras derived equivalent to $\uEnd_\cE(M)$ are precisely the stable endomorphism algebras of maximal  rigid objects in $\ucE$. In particular, there are only finitely many such algebras.
\end{cor}
Moreover, the standard derived equivalences between these algebras can be thought of simply as mutation sequences (or equivalently, paths in the mutation graph of rigid objects) in a way made precise in Corollary \ref{standard}. In this way, we are able to produce a `picture' of the derived equivalence class of these algebras (see Example \ref{ex2}). 
\subsection{Geometric Corollaries}
Although the algebraic conditions in Theorem \ref{mainresultintro} may seem restrictive, a large source of examples can be found in the Homological Minimal Model Programme (see \S \ref{geometricapplication} or \cite{HomMMP}). This is an algebraic approach to the classical Minimal Model Programme and our main application of Theorem \ref{mainresultintro} is to study the derived equivalence classes of \emph{contraction algebras} appearing in this setting. 

Given an algebraic variety $X$, the goal of the Minimal Model Programme is to find and study certain birational maps $f \colon Y \to X$, known as \textit{minimal models}. Minimal models of $X$ are not unique but Kawamata \cite{[K]} showed any two minimal models are connected by a sequence of codimension two modifications called \textit{flops}.
For this reason, maps which give rise to flops, known as \textit{flopping contractions}, are of particular interest.

For a $3$-fold flopping contraction $f \colon X \to X_\con$, Donovan--Wemyss introduced an algebraic invariant, known as the \textit{contraction algebra} of $f$. If the map contracts a single curve, this finite dimensional algebra recovers all the previously known numerical invariants \cite{DefFlops, invsing} and has even been used to prove that those numerical invariants can not classify flops completely \cite{GVinv}.
It is further conjectured by Donovan--Wemyss that, in the setting of smooth minimal models, the contraction algebra, or more precisely its derived category, completely determines the geometry.  

\begin{conj} \label{conjwemyss}
Suppose that $f \colon X \to \Spec R$ and $g \colon Y \to \Spec S$ are smooth minimal models of complete local isolated cDV singularities with associated contraction algebras $A_\con$ and $B_\con$. Then $R \cong S$ if and only if $A_\con$ and $B_\con$ are derived equivalent.
\end{conj}
 
The `only if' direction is already known to be true by combining results from \cite{HomMMP} with the main result of \cite{Dugas}, but the `if' direction remains a key open problem in the Homological Minimal Model Programme. This conjecture is the main motivation for studying the derived equivalence classes of contraction algebras.

The connection between flops and cluster-tilting theory has been explored in \cite{HomMMP}, where it is shown that, for certain flopping contractions $f \colon X \to \Spec R$, the contraction algebra is the stable endomorphism algebra of a rigid object in some Frobenius category associated to $R$. In particular, the setting there satisfies the conditions of Theorem \ref{mainresultintro} and thus, we obtain the following result.   

\begin{thm}[Theorem \ref{geomres}]
Suppose that $f\colon X \to \Spec R$ is a 3-fold flopping contraction where $R$ is complete local and $X$ has at worst Gorenstein terminal singularities. Writing $A_\con$ for the associated contraction algebra, the following statements hold.
\begin{enumerate}
\item The basic algebras in the derived equivalence class of $A_\con$ are precisely the contraction algebras of flopping contractions obtained by iterated flops from $f$.
\item There are only finitely many basic algebras in the derived equivalence class of $A_\con$.
\end{enumerate}
\end{thm}
Minimal models are a special case where the contraction algebras correspond to maximal rigid objects, and this leads to the following analogue of Corollary \ref{maxrigintro}. In many ways, this corollary can be viewed as the 3-fold analogue of \cite[5.1]{AiharaMiz}.
\begin{cor}[Corollary \ref{geommaxrigid}]
Suppose that $f\colon X \to \Spec R$ is a minimal model of a complete local isolated cDV singularity and write $A_\con$ for the associated contraction algebra. Then the following statements hold.
\begin{enumerate}
\item The basic algebras in the derived equivalence class of $A_\con$ are precisely the contraction algebras of the minimal models of $\Spec R$.
\item There are only finitely many basic algebras in the derived equivalence class of $A_\con$.
\end{enumerate}
\end{cor}

This provides a large class of algebras with finite derived equivalence classes. Even in the more restricted setting of Conjecture \ref{conjwemyss}, combining \cite[5.5]{morrison} and \cite[4.10(2)]{HomMMP} shows that the quivers of the associated contraction algebras form seven infinite families. 

In the process of proving the above, we also establish the following.
\begin{thm}[Theorem \ref{evidence}]
Suppose that $f \colon X \to \Spec R$ and $g \colon Y \to \Spec S$ are minimal models of complete local isolated cDV singularities with associated contraction algebras $A_\con$ and $B_\con$. If $A_\con$ and $B_\con$ are derived equivalent then there is a bijection 
\begin{align*}
\{\text{minimal models of $\Spec R$}\} \leftrightarrow \{\text{minimal models of $\Spec S$}\}.
\end{align*}
Further, the bijection preserves both iterated flops and contraction algebras. 
\end{thm}
This gives the first concrete evidence towards Conjecture \ref{conjwemyss}, as it shows that $\Spec R$ and $\Spec S$ must at least have the same number of minimal models and further that the simple flops graphs of their minimal models must be isomorphic.  

\subsection{Conventions} Throughout, $k$ will denote an algebraically closed field of characteristic zero. For a ring $A$, we denote the category of finitely generated right modules by $\mod A$. For $M \in \mod A$, we let $\add M$ be the full subcategory consisting of summands of finite direct sums of copies of $M$ and we let $\proj A \colonequals \add A$ be the category of finitely generated projective modules. Finally, $\Kb(\proj A)$ will denote the homotopy category of bounded complexes of finitely generated projectives and $\Db(A) \colonequals \Db(\mod A)$ will denote the bounded derived category of $\mod A$. 

\subsection{Acknowledgments} 
While this work was carried out, the author was a student at the University of Edinburgh and the material contained in this paper formed part of her PhD thesis. The author would like to thank her supervisor Michael Wemyss for his helpful guidance and the Carnegie Trust for the Universities of Scotland for their financial support. Furthermore, the author would like to thank the referees for their careful reading and valuable suggestions, and the Max Planck Institute for Mathematics in Bonn for its hospitality and financial support during the editing process.

\section{Preliminaries}

\subsection{Silting Theory}

In this subsection we recall silting and tilting theory for a finite dimensional $k$-algebra $\Lambda$. Note that for such a $\Lambda$, the bounded homotopy category $\Kb(\proj \Lambda)$ is a $k$-linear, Hom-finite, Krull-Schmidt triangulated category.
\begin{definition}
A complex $P \in \Kb(\proj \Lambda)$ is called:
\begin{enumerate}
\item \emph{presilting} (respectively \emph{pretilting}) if $\Hom_{\Lambda}(P,P[n])=0$ for all $n >0$ (respectively for all $n \neq 0$).
\item \emph{silting} (respectively \emph{tilting}) if $P$ is presilting (respectively pretilting) and further the smallest full triangulated subcategory of $\Kb(\proj \Lambda)$ containing $P$ and closed under forming direct summands is $\Kb(\proj \Lambda)$.
\end{enumerate}
\end{definition}
We will write $\silt \Lambda$ (respectively $\tilt \Lambda$) for the set of isomorphism classes of basic silting (respectively tilting) complexes in $\Kb(\proj \Lambda)$. Note that $\Lambda \in \tilt \Lambda$ and that all $T \in \silt \Lambda$ have the same number of indecomposable summands \cite[2.8]{siltingmutation}. 

\subsubsection{Derived Equivalences} \label{derivedprelim}Tilting objects are of interest due to the following well-known theorem, connecting them with derived equivalences.

\begin{thm} \cite{morita}\label{rickards}
For each tilting complex $T \colonequals \bigoplus_{i=1}^n T_i$ in $\Kb(\proj \Lambda)$ there exists a triangle equivalence $\Db(\Lambda) \to \Db(\End_{\Lambda}(T))$ sending $T_i \mapsto \Hom_\Lambda(T, T_i)$. Moreover, a basic finite dimensional algebra $\Gamma$ is derived equivalent to $\Lambda$ if and only if there exists $T \in \tilt \Lambda$ such that $\Gamma \cong \End_{\Lambda}(T)$. 
\end{thm}

If we restrict to standard derived equivalences, the connection with tilting theory becomes even stronger. 
\begin{definition}
A triangle equivalence $F \colon \Db(\Lambda) \to \Db(\Gamma)$ is called \emph{standard} if it is isomorphic to 
\begin{align*}
\RHom_{\Lambda}(\tT,-)
\end{align*}
for some complex $\tT$ of $\Gamma$-$\Lambda$-bimodules. In this case, we call $\tT$ a \emph{two-sided tilting complex}.
\end{definition}
It is shown in \cite[4.1]{standard} that the composition of standard equivalences is again standard and further, the inverse of a standard equivalence is also standard. Moreover, for any tilting complex $T \in \Db(\Lambda)$, it is shown in \cite{keller} that there exists a two-sided tilting complex $\tT$ such that $T \cong \tT$ in $\Db(\Lambda)$. This induces a standard equivalence
\begin{align*}
\RHom_\Lambda(\tT, -) \colon \Db(\Lambda) &\to \Db(\End_\Lambda(T))
 \end{align*} 
which maps $T \mapsto \End_\Lambda(T)$. It was further shown in \cite[2.1]{kellerhomotopy} that such a $\tT$ is unique in a suitable sense. For our purposes, the following suffices.
\begin{proposition}\cite[2.3]{rouquier} \label{rouq}
Suppose that $\Lambda$, $\Gamma$ and $\Gamma'$ are $k$-algebras and that $T$ (respectively $T'$) is a two-sided tilting complex of $\Lambda$-$\Gamma$-bimodules (respectively $\Lambda$-$\Gamma'$-bimodules) with $\End_\Lambda(T) \cong \End_\Lambda(T')$. Then $T_\Lambda \cong T_\Lambda'$ if and only if there exists an isomorphism $\upgamma \colon \Gamma \to \Gamma'$ such that 
\begin{align*}
T \cong {}_\upgamma \Gamma' \otimes_{\Gamma'} T'
 \end{align*} 
 in the derived category of $\Gamma$-$\Lambda$ bimodules.
\end{proposition}
In particular, for any tilting complex there is a unique (up to algebra isomorphism) standard equivalence induced by the tilting complex.
\subsubsection{Mutation}
Although silting complexes do not necessarily induce derived equivalences, the advantage of silting over tilting complexes is that they have a well-behaved notion of mutation. To define this we require the following.

Suppose that $\cD$ is an additive category and $\cS$ is a class of objects in $\cD$.
\begin{enumerate}
\item A morphism $f \colon X \to Y$ is called a \emph{right $\cS$-approximation} of $Y$ if $X \in \cS$ and the induced morphism $\Hom(Z,X) \to \Hom(Z,Y)$ is surjective for any $Z \in \cS$.
\item A morphism $f \colon X \to Y$ is said to be \emph{right minimal} if for any $g \colon X \to X$ such that $f \circ g= f$, then $g$ must be an isomorphism.
\item A morphism $f \colon X \to Y$ is a \emph{minimal right $\cS$-approximation} if $f$ is both right minimal and a right $\cS$-approximation of $Y$.
\end{enumerate}
There is also the dual notion of a (minimal) left $\cS$-approximation.
\begin{definition} \cite[2.31]{siltingmutation}\label{mutdef}
Let $P \in \Kb(\proj \Lambda)$ be a basic silting complex for $\Lambda$, and write $P \colonequals \bigoplus\limits_{i=1}^n P_i$ where each $P_i$ is indecomposable. Consider a triangle
\begin{align*}
P_i \xrightarrow{f} P' \to Q_i \to P_i[1] 
\end{align*} 
where $f$ is a minimal left $\add(P/P_i)$-approximation of $P_i$. Then $\upmu_i(P)\colonequals (P/P_i) \oplus Q_i$ is also a silting complex, known as the \emph{left mutation} of $P$ with respect to $P_i$. 
\end{definition}
Right mutation is defined dually and is denoted $\upmu_i^{-1}$ as it is inverse to left mutation \cite[2.33]{siltingmutation}\footnote{Note that this notation for left/right follows \cite{siltingmutation} and \cite{AiharaMiz}, but is opposite to the conventions in \cite{[AIR]}.}. For general $\Lambda$, the mutation of a tilting complex may not be a tilting complex. However, it is well known (see e.g. \cite[2.8]{siltingmutation}) that if $\Lambda$ is a symmetric algebra, i.e. $\Lambda \cong \Hom_k(\Lambda, k)$ as $\Lambda$-$\Lambda$ bimodules, then any silting complex is a tilting complex and hence any mutation of a tilting complex is again a tilting complex. This will be the case in our setting later. 

To help control mutations, Aihara--Iyama introduced a partial order on $\silt \Lambda$ \cite{siltingmutation}, which generalised the partial order on tilting modules from \cite{rs}.
\begin{definition}
Let $P$ and $Q$ be silting complexes for $\Lambda$. If $\Hom_\Lambda(P,Q[i])=0$ for all $i >0$, then we say $P \geq Q$. Further,  we write $P > Q$ if $P \geq Q$ and $P \ncong Q$.
\end{definition}
This order can determine whether two silting complexes are related by mutation. 

\begin{thm}\cite[3.5]{tiltingconnectedness} \label{leftconnected}
Suppose that $T,U$ are two basic silting complexes for $\Lambda$ with $T \geq U$. If there are only finitely many basic silting complexes $P$ such that $T \geq P \geq U$, then $U$ is obtained by iterated left mutation from $T$ or equivalently, $T$ is obtained by iterated right mutation from $U$.
\end{thm}

The following result, which is implicit in the literature, will be useful for tracking silting complexes through derived equivalences. 
\begin{lemma} \label{tracking}
Let $\Lambda$ and $\Gamma$ be finite dimensional $k$-algebras and $F \colon \Db(\Lambda) \to \Db(\Gamma)$ be a triangle equivalence. Then the following statements hold.
\begin{enumerate}
\item $F$ maps silting complexes to silting complexes.
\item $F$ preserves the silting order.
\item If $P$ is a silting complex for $\Lambda$, then 
\begin{align*}
F(\upmu_i(P)) \cong  \upmu_i(F(P)) \quad \text{and} \quad F(\upmu^{-1}_i(P)) \cong  \upmu^{-1}_i(F(P)).
\end{align*}
\end{enumerate}
\end{lemma}
\begin{proof} 

\begin{enumerate}[leftmargin=0cm,itemindent=.6cm,labelwidth=\itemindent,labelsep=0cm,align=left]
\item A standard result of Rickard \cite[6.2]{morita} states that $F$ restricts to an equivalence
\begin{align*}
\Kb(\proj \Lambda) \to \Kb(\proj \Gamma). 
\end{align*}
Using that equivalences are fully faithful and must preserve generators easily establishes that the properties of silting complexes are also preserved. 
\item Again, this is a simple consequence of the fully faithful property of equivalences. 
\item Since $F$ is a triangle equivalence, it preserves both triangles and minimal left/right approximations. Thus, the exchange triangle defining $\upmu_i(P)$ is mapped under $F$ to the exchange triangle defining $\upmu_i(F(P))$ and so the first statement, and similarly the second, follow.
\qedhere
\end{enumerate}
\end{proof}

\subsubsection{Two-Term Silting Complexes}
Two-term silting complexes are an important class of silting complexes that have connections with cluster-tilting theory. We recall their definition here.
\begin{definition} \label{definetwoterm}
A presilting complex $P \in \Kb(\proj \Lambda)$ is called \emph{two-term} if the terms are zero in every degree other than $0$ and $-1$, or equivalently by \cite[2.9]{tiltingconnectedness}, if $\Lambda \geq P \geq \Lambda[1]$. 
\end{definition}
We denote the set of isomorphism classes of basic two-term silting (respectively presilting, tilting) complexes as $\twosilt \Lambda$ (respectively $\presilt \Lambda$, $\twotilt \Lambda$). 
The following shows that mutation of two-term silting complexes is particularly well behaved.
\begin{proposition}\cite[3.8]{[AIR]} \label{twotermmutation}
Suppose $P$ and $Q$ are basic two-term silting complexes for $\Lambda$. Then $P$ and $Q$ are related by a single mutation if and only if they differ by exactly one indecomposable summand. 
\end{proposition}

\subsubsection{Silting-Discreteness}
This subsection recalls silting-discrete algebras, first introduced in \cite{tiltingconnectedness}.  
\begin{definition}
A finite dimensional algebra $\Lambda$ is said to be \emph{silting-discrete} if for any $P \in \silt \Lambda$ and any $\ell >1$, the set
\begin{align*}
 \lsilt_P \Lambda \colonequals \{ \ T \in \silt  \Lambda \mid P \geq T \geq P[\ell-1] \ \}
 \end{align*} 
is finite. Further, $\Lambda$ is called \emph{$2$-silting-finite} if $\twosilt_P \Lambda$ is a finite set for any $P \in \silt \Lambda$.

\end{definition}
Notice that $\twosilt_\Lambda \Lambda =\twosilt \Lambda$. The key advantage of silting-discrete algebras is the following, first shown in \cite[3.9]{tiltingconnectedness} but provided here with a proof for convenience.
\begin{proposition} \label{tiltingconnected}
If $\Lambda$ is a silting-discrete finite dimensional algebra then any silting complex $T \in \Kb(\proj \Lambda)$ can be obtained from $\Lambda$ by finite iterated mutation. 
\end{proposition}
\begin{proof}
Choose a silting complex $T$. By \cite[2.9]{tiltingconnectedness}, there exists integers $m>n$ such that $\Lambda[n] \geq T \geq \Lambda[m]$. We now split the proof into two cases: when $n \geq 0$, and when $n<0$.

If $n \geq 0$, the set $\{Q \in \silt \Lambda \mid \Lambda \geq Q \geq T \} \subseteq (m+1)$-$\silt_\Lambda \Lambda$ and hence is finite by the silting-discrete assumption. Thus, by Theorem \ref{leftconnected}, $T$ is obtained from $\Lambda$ by iterated left mutation.

If $n<0$, the set $(1-n)$-$\silt_{\Lambda[n]} \Lambda = \{Q \in \silt \Lambda \mid \Lambda[n] \geq Q \geq \Lambda\}$ is finite as $\Lambda$ is silting-discrete. Using Theorem \ref{leftconnected}, this shows $\Lambda[n]$ can be obtained by iterated right mutation from $\Lambda$. Further, $\{Q \in \silt \Lambda \mid \Lambda[n] \geq Q \geq T \} \subseteq (m-n+1)$-$\silt_{\Lambda[n]} \Lambda$ is also finite, showing $T$ can be obtained from $\Lambda[n]$ by iterated left mutation, again using Theorem \ref{leftconnected}. Combining these mutation sequences proves the result. 
\end{proof}
The following result establishes equivalent conditions for an algebra to be silting-discrete.
\begin{thm}\cite[2.4]{AiharaMiz} \label{tiltingdiscrete}
Let $\Lambda$ be a finite dimensional algebra. Then the following are equivalent:
\begin{enumerate}
\item $\Lambda$ is silting-discrete.
\item $\Lambda$ is $2$-silting-finite.
\item $\twosilt_P \Lambda$ is a finite set for any silting complex $P$ which is given by iterated left mutation from $\Lambda$.
\end{enumerate}
\end{thm}

\subsection{Cluster-Tilting Theory}
Throughout this subsection, $\cC$ will denote a $k$-linear Hom-finite, Krull-Schmidt, 2-Calabi-Yau triangulated category with shift functor $\Sigma$. The property 2-Calabi-Yau (2-CY) means that there are bifunctorial isomorphisms
\begin{align*}
\Hom_\cC(M,N[2]) \cong D\Hom_\cC(N,M)
\end{align*} 
for all $M,N \in \cC$ where $D \colonequals \Hom_k(-,k)$. 
\begin{definition} \label{definerigid}
Let $M \in \cC$. 
\begin{enumerate}
\item $M$ is called \emph{rigid} if $\mathrm{Hom}_{\cC}(M,\Sigma M)=0$.
\item $M$ is called \emph{maximal rigid} if $M$ is rigid and if $M \oplus X$ is rigid for some $X \in \cC$, then $X \in \add(M)$.
\end{enumerate}
\end{definition}
Write $\rig \cC$ for the set of basic rigid objects in $\cC$ and write $\mrig \cC$ for the set of basic maximal rigid objects (both taken up to isomorphism). Further, if $\cS$ is a collection of objects in $\cC$, we write $\rig \cS$ for the intersection $\cS \cap \rig \cC$. The mutation of these objects is defined similarly to the mutation of silting complexes.
\vspace{-1mm}
\begin{definition} \label{mutaterigid}
Suppose that $M \colonequals \bigoplus\limits_{i=1}^n M_i$ is a basic rigid object in $\cC$ with each $M_i$ indecomposable. Consider a triangle
\begin{align*}
M_i \xrightarrow{f_i} V_i \to N_i \to \Sigma M_i
\end{align*} 
where $f_i$ is a minimal left $\add(M/M_i)$-approximation of $M_i$. Then $\upnu_i(M)\colonequals (M/M_i) \oplus N_i$ is also a rigid object, known as the \emph{left mutation} of $M$ with respect to $M_i$. We call the triangle an \emph{exchange triangle}.
\end{definition}
Right mutation is defined dually and we denote it by $\upnu_i^{-1}$.  As with silting complexes, right and left mutation are inverse operations.
\vspace{-1mm}
\begin{lemma} \label{rightleftmut}
For any rigid object $M \colonequals \bigoplus\limits_{i=1}^n M_i \in \cC$ and any $i$, $\upnu_i \upnu_i^{-1} M \cong M \cong \upnu_i^{-1} \upnu_i M$.
\end{lemma}
\vspace{-3.7mm}
\begin{proof} For lack of a reference, we sketch the proof.
To show this, it is enough to show that, if $f_i$ is a minimal left $\add(M/M_i)$-approximation of $M_i$ in the exchange triangle 
\begin{align} 
M_i \xrightarrow{f_i} V_i \xrightarrow{g_i} N_i \to \Sigma M_i,  \label{exchange}
\end{align} 
then $g_i$ is a minimal right $\add(M/M_i)$-approximation of $N_i$ and vice versa. 

This is very similar to \cite[5.7, 5.8]{GLS}: begin by assuming  $f_i$ is a minimal left $\add(M/M_i)$-approximation of $M_i$. Applying $\Hom_{\cC}(M/M_i,-)$ to the exchange triangle \eqref{exchange} and using that $M$ is rigid gives an exact sequence
\begin{align*}
\Hom_{\cC}(M/M_i,V_i) \xrightarrow{g_i \circ -}  \Hom_{\cC}(M/M_i,N_i) \to 0.
\end{align*}
This shows that $g_i$ is a right $\mathrm{add}(M/M_i)$-approximation. There is an isomorphism 
\begin{align*}
\left( V_i \xrightarrow{g_i} N_i \right) \cong \Big( W \oplus Z \xrightarrow{(g',0)} N_i \Big)
\end{align*}
where $g'$ is right minimal and $W, Z \in \add(M/M_i)$. Completing these maps to triangles gives
\vspace{-1.5mm}
\begin{align*}
\left( M_i \xrightarrow{f_i} V_i \xrightarrow{g_i} N_i \to \Sigma M_i \right) \cong 
\left( \Big( W' \xrightarrow{f'} W \xrightarrow{g'} N_i \to \Sigma M_i \Big) \oplus \left( Z \xrightarrow{\id} Z \xrightarrow{0} 0 \to \Sigma Z \right)  \right) 
\end{align*}
and hence by the uniqueness of cocones, $M_i \cong W' \oplus Z$. Since $M_i$ is indecomposable, either $W'$ or $Z$ must be zero. If $W'$ is zero, $M_i \cong Z$ and so $M_i \in \mathrm{add}(Z) \subseteq \mathrm{add}(M/M_i)$ which is a contradiction. Hence, $Z \cong 0$ and so $g_i$ is right minimal, as required. The other direction is a dual argument. 
\end{proof}
For maximal rigid objects $M$ and $N$, it is shown in \cite[3.3]{maxrigid}, generalising \cite[5.3]{yoshino} for cluster-tilting objects, that $M$ is a mutation of $N$ if and only if $M$ and $N$ differ by exactly one indecomposable summand. An easy consequence of this is that left and right mutation must coincide in this case.

For a rigid object $M \in \cC$, we will write $\mut(M)$ for the collection of basic rigid objects which can be obtained from $M$ by iterated left or right mutation. 
The relationship between cluster-tilting theory and silting theory relies on the following subcategory of $\cC$.
\begin{definition}
Given $M \in \cC$, define $M * \Sigma M$ to be the full subcategory of $\cC$ consisting of the objects $N$ such that there exists a triangle
\begin{align}
M_1 \xrightarrow{f} M_0 \xrightarrow{g} N \to \Sigma M_1 \label{tstar}
\end{align}
where $M_1,M_0 \in \mathrm{add}(M)$.
\end{definition}
It is easy to show that if $M$ is rigid, then for any triangle such as \eqref{tstar}, $g$ is a right $\add(M)$-approximation and further, by possibly changing $M_0$ and $M_1$, we can choose $g$ to be minimal. 
The following theorem is a slight generalisation of \cite[4.7]{[AIR]}, similar to that of \cite[3.2]{rigid}. 
\begin{thm}\cite[4.7]{[AIR]} \label{ctbij}
Let $\cC$ be a $k$-linear Hom-finite, Krull-Schmidt, 2-CY triangulated category and $M$ be a rigid object of $\cC$. If $\Lambda \colonequals \End_\cC(M)$, there is a bijection 
\begin{align*}
\rig (M*\Sigma M)  \longleftrightarrow \presilt \Lambda
\end{align*}
which preserves the number of summands. For a rigid object $N$, consider the triangle
\begin{align*}
M_1 \xrightarrow{f} M_0 \xrightarrow{g} N \to \Sigma M_1
\end{align*}
such that $g$ is a minimal right $\mathrm{add}(M)$-approximation. Applying $\Hom_\cC(M,-)$ to $f$ gives the corresponding 2-term presilting complex.
\end{thm}
\begin{proof}
As the proof is so similar to \cite[4.7]{[AIR]}, just replacing $\cC$ with $M*\Sigma M$ throughout, we only give a sketch of the proof and highlight where adaptations are required. 

It is well known (see e.g.\ \cite[2.3]{krause}) that the functor $\Hom_\cC(M,-)$ induces an equivalence $\add(M) \xrightarrow{\sim} \proj \Lambda$. This can be used exactly as in \cite[4.6]{[AIR]} to show that the map described in the statement gives a bijection
\begin{align*}
\{\text{objects in $M*\Sigma M$}\} \quad \longleftrightarrow \quad \{\text{complexes $P_{-1} \to P_0$ in $\Kb(\proj \Lambda)$}\}
\end{align*}
where both sides are taken up to isomorphism. The claim that a rigid object $N \in M*\Sigma M$ is sent to a presilting complex then follows exactly as in \cite[4.7]{[AIR]}, making repeated use of the equivalence $\add(M) \xrightarrow{\sim} \proj \Lambda$. 

Finally, given a two-term presilting complex $P$, this is necessarily of the form
\begin{align*}
\Hom_\cC(M,M_1) \xrightarrow{f \circ -} \Hom_\cC(M,M_0)
\end{align*} 
for some $M_1 \xrightarrow{f} M_0$ in $\add(M)$ and the object in $M *\Sigma M$ associated to $P$ is $N \colonequals \mathrm{cone}(f)$. The proof that $N$ is rigid relies on showing that if a map $M_0 \xrightarrow{p} \Sigma^2M_1$ satisfies $\Hom_\cC(M,p)=0$, then $p=0$. In \cite{[AIR]} they use an equivalence $\cC/[\Sigma M] \xrightarrow{\sim} \mod \Lambda$ which holds for any cluster-tilting object $M$, and which was generalised in \cite[2.1]{rigid}  to an equivalence $(M*\Sigma M)/[\Sigma M] \xrightarrow{\sim} \mod \Lambda$ for any rigid object $M$. However, this is not directly applicable here as it is unclear that $\Sigma^2M_1 \in M*\Sigma M$. Instead, we can appeal to the general result (see e.g.\ \cite[VI.3.1]{ASS}) that for any $M' \in \add(M)$ and any $N \in \cC$, the functor $\Hom_\cC(M,-)$ induces an isomorphism
\begin{align*}
\Hom_\cC(M',N) \cong \Hom_\Lambda(\Hom_\cC(M,M'), \Hom_\cC(M,N)). 
\end{align*} 
The proof then proceeds exactly as stated in \cite[4.7]{[AIR]}.
\end{proof}
\begin{remark} \label{maxrigremark}
\begin{enumerate}[leftmargin=0cm,itemindent=.6cm,labelwidth=\itemindent,labelsep=0cm,align=left]
\item Since the bijection preserves the number of summands, rigid objects in $M*\Sigma M$ with the same number of summands as $M$ must correspond to silting complexes. 
\item If $M$ is maximal rigid, it is shown in \cite[2.5]{maxrigid} that all rigid objects of $\cC$ lie in $M*\Sigma M$ and so this bijection restricts to 
\begin{align*}
\mrig \cC \longleftrightarrow \twosilt \Lambda.
\end{align*}
Further, the bijection respects mutation in this case as mutation on both sides corresponds to differing by exactly one indecomposable summand.
\item If $\cC$ has only finitely many maximal rigid objects, Theorem \ref{ctbij} can be used along with \cite[2.38, 3.9]{[AIR]}, to show that, for any maximal rigid object $M$, $\mut(M) = \mrig \cC$. This is exactly as in \cite[4.9]{[AIR]}, but where we replace cluster-tilting with maximal rigid.
\end{enumerate}
\end{remark}
Theorem \ref{ctbij} can also be used to show the following, which will show that if there are only finitely many basic maximal rigid objects in some category, there must also be only finitely many basic rigid objects.
\begin{lemma} \label{rigsummand}
Let $\cC$ be a $k$-linear Hom-finite, Krull-Schmidt, 2-CY triangulated category. If there exists a maximal rigid object in $\cC$, then any basic rigid object $M \in \cC$ is a direct summand of some basic maximal rigid object. 
\end{lemma}
\begin{proof}
Let $N \in \cC$ be a basic maximal rigid object and write $\Lambda \colonequals \End_\cC(N)$. Let $\phi$ denote the corresponding bijection from Theorem \ref{ctbij}. Now take any basic rigid object $M \in \cC$. By \cite[2.5]{maxrigid}, $M$ is contained in $\rig (N*\Sigma N)$ and thus $\phi(M) \in \presilt \Lambda$. Hence, by \cite[3.1]{siltingtheorem}, there exists $P \in \presilt \Lambda$ such that $\phi(M) \oplus P$ is a two-term silting complex for $\Lambda$. Mapping back across the bijection, and using Remark \ref{maxrigremark}(2), shows that $M \oplus \phi^{-1}(P)$ is maximal rigid object which gives the result.
\end{proof}


\section{The Derived Equivalence Class}
\subsection{Main Results}
The following setup will be used throughout this section.
 \begin{setup} \label{setup2}
Suppose $\cE$ is a Frobenius category such that its stable category $\ucE$ is a $k$-linear, Hom-finite, Krull-Schmidt, 2-CY triangulated category with shift functor $\Sigma$. Additionally assume that:
\begin{itemize}
\item $\Sigma^2$ is isomorphic to the identity functor on $\ucE$.
\item $\ucE$ has at least one but only finitely many basic maximal rigid objects.
\end{itemize} 
 \end{setup}
 \begin{remark} \label{finiterigid}\begin{enumerate}[leftmargin=0cm,itemindent=.6cm,labelwidth=\itemindent,labelsep=0cm,align=left]
\item The first additional assumption means $\ucE$ is a $0$-Calabi-Yau category, but as we wish to view $\ucE$ as a category with cluster-tilting theory plus some additional conditions, we write the assumptions as above.
\item By Lemma \ref{rigsummand}, the second assumption further implies that $\ucE$ has only finitely many basic rigid objects. 
\end{enumerate}
 \end{remark}

  A large source of examples of this setup will come from a geometric setting, described in \S \ref{geometricapplication}. For an infinite family of examples, see \S \ref{geometricexample}. 
  
The first additional assumption gives two important results, the first of which is the following. This was first proved in the setting of hypersurface singularities, but the proof works generally.

\begin{proposition} \cite[7.1]{symmetric} \label{symmalg}
Let $\cC$ be a $k$-linear, Hom-finite, 2-CY triangulated category with shift functor $\Sigma$ such that $\Sigma^2 \cong \mathrm{id}$. Then for any $N \in \cC$, the algebra $\Lambda \colonequals\End_{\cC}(N)$ is a symmetric algebra i.e. $\Lambda \cong D\Lambda$ as bimodules.
\end{proposition}
\begin{proof}
For any $M,N \in \cC$, $\Hom_{\cC}(M,N)$ has the structure of an $\End_{\cC}(M)$-$\End_{\cC}(N)$-bimodule. Further, the bifunctoriality of the isomorphisms
\begin{align*}
\Hom_{\cC}(M,N) \cong D\Hom_{\cC}(N,\Sigma^2 M)
\end{align*}
coming from the 2-CY property ensure they are isomorphisms of $\End_{\cC}(M)$-$\End_{\cC}(N)$-bimodules. Taking $M=N$ and using that $\Sigma^2 \cong \mathrm{id}$ gives the desired result.
\end{proof}
The key advantage of this result is that it means any silting complex for these endomorphism algebras will in fact be a tilting complex. In particular, the left or right mutation of any tilting complex at any summand, as in Definition \ref{mutdef}, will again be a tilting complex. This leads to the second important consequence of our assumption on $\Sigma$.

\begin{proposition}\label{derivedequiv}
With the setup of \ref{setup2}, let $M \colonequals \bigoplus_{i=1}^n M_i$ be a rigid object of $\ucE$ and write $\Lambda:=\uEnd_\cE(M)$. Mutate $M$ at the summand $M_i$ via the exchange sequence \eqref{exchange} and consider the two-term complex
\begin{align*}
P \colonequals \big( 0 \to \bigoplus_{j \neq i} \uHom_{\cE}(M,M_j) \big) \oplus \big( \uHom_{\cE}(M,M_i) \xrightarrow{f_i \circ -} \uHom_{\cE}(M,V_i) \big).
\end{align*}
Then $P$ is isomorphic to the tilting complex $\upmu_i \Lambda$ and there is a ring isomorphism 
\begin{align*}
\End_\Lambda(P) \cong \uEnd_\cE(\upnu_i M).
\end{align*} 
\end{proposition}
\begin{proof}
The fact that $P$ is a tilting complex and has the required endomorphism ring will follow directly from \cite[4.1]{Dugas} if we can show that the conditions there hold. In particular, we need to show that, for any $j \in \mathbb{Z}$, the maps
\begin{align*}
\uHom_\cE(V_i, \Sigma^j(M/M_i)) \xrightarrow{- \circ f_i} \uHom_\cE(M_i, \Sigma^j(M/M_i)) 
\end{align*}
and 
\begin{align*}
\uHom_\cE(\Sigma^j(M/M_i), V_i) \xrightarrow{g_i \circ -} \uHom_\cE(\Sigma^j (M/M_i), N_i)
\end{align*} 
are surjective. Using the assumption $\Sigma^2 \cong \id$, we need only consider $j=0,1$. When $j=0$ this follows as both $f_i$ and $g_i$ are $\add(M/M_i)$-approximations ($f_i$ by definition and $g_i$ by Lemma \ref{rightleftmut}). When $j=1$, rigidity of $M$ and $\upnu_iM$ show that all the terms are zero and hence the maps are surjective as required. Finally, we see that $P$ must be $\upmu_i \Lambda$ by Proposition \ref{twotermmutation}, as $P$ is a two-term silting complex differing from $\Lambda$ by precisely the $i^{th}$ summand.
\end{proof}

Combining Proposition \ref{derivedequiv} and Theorem \ref{rickards} shows that, for any basic rigid object $M$, there is a derived equivalence
\begin{align*}
\Db(\uEnd_\cE(M)) &\to \Db(\uEnd_\cE(\upnu_i M)),\\
\upmu_i \uEnd_\cE(M) &\mapsto \uEnd_\cE(\upnu_i M).
\end{align*}
By iterating this result, it is easy to see that the stable endomorphism algebras of the elements in the set $\mut(M)$ are all derived equivalent algebras. 

Using the above two results allows us to mutate freely and to investigate how mutation of rigid objects and mutation of tilting complexes interact. Under the setup of \ref{setup2}, fix a basic rigid object $M \colonequals \bigoplus_{i=1}^n M_i\in \ucE$ and let $\Lambda \colonequals \uEnd_\cE(M)$. Note that each sequence of mutations from $M$ produces a rigid obect $N \in \mut(M)$ along with an induced ordering of the summands of $N$. To record this data, we introduce the following directed graph which we denote $\umut(M)$: 
\begin{itemize}
\item the vertices are ordered tuples $\overline{N}=(N_1, \dots, N_n)$, where $N \colonequals \bigoplus_{i=1}^n N_i \in \mut(M)$ and the order of the summands is induced by some sequence of mutations from $M$. Two tuples $\overline{N}$ and $\overline{N}'$ are identified precisely when $N_i \cong N_i'$ for each $i \in \{1, \dots, n\}$;
\item there is an arrow $s_i \colon \overline{N} \to \overline{N}'$ if $N$ and $N'$ differ in precisely the $i^{th}$ summand, and $N' \cong \upnu_iN$. 
\end{itemize}

Note that each vertex has $n$ arrows $s_1, \dots, s_n$ starting at the vertex (corresponding to left mutations) and $n$ arrows $s_1, \dots, s_n$ incident at the vertex (corresponding to right mutations). 

Note that if $M$ and $N$ are related by mutation, then $\umut(M)$ and $\umut(N)$ will be isomorphic graphs. In particular, for any maximal rigid object, the mutation graph will be the same and so we just denote this $\umrig(\cE)$.

The idea behind the results in this paper is that combinatorial paths in $\umut(M)$ will control not only mutation of rigid objects but also tilting complexes of $\Lambda$.

\begin{definition}
A \emph{path} in $\umut(M)$ is a symbol $s_{i_m}^{\upvarepsilon_{m}} \dots s_{i_1}^{\upvarepsilon_{1}}$ with $i_1, \dots, i_m \in \{1, \dots n\}$ and $\upvarepsilon_{i} \in \{-1,1\}$, along with a specified starting vertex $\overline{N}$. The path $s_i$ starting at $\overline{N}$ should be thought of as the path travelling along arrow $s_i$ from vertex $\overline{N}$ and the path $s_i^{-1}$ should be thought of travelling backwards along the arrow $s_i$ incident to $\overline{N}$. Longer paths are composed right to left as with function composition. A path is called \emph{positive} if all the $\upvarepsilon_i$ equal $1$.
\end{definition}

\begin{notation}
When the ordering is clear from the context, such as when $N$ is defined via a sequence of mutations from our fixed $M$, we will abuse notation and drop the overline notation for the corresponding vertex of $\umut(M)$. Moreover, ``a path starting at $M$" will always refer to a path starting from $\overline{M}=(M_1, \dots, M_n)$, where $M\colonequals \bigoplus_{i=1}^n M_i$ is the fixed ordering used to define $\umut(M)$.
\end{notation}

\begin{example} \label{ex1}
In the geometric setting introduced later, there exists an example of $\cE$ where $\umrig(\cE)$ is the figure on the left below; we write $M_{{i_n} \dots {i_1}} \colonequals \upnu_{i_n} \dots \upnu_{i_1}M$ to ease notation and note that, in this example, $M_{2121} \cong M_{1212}$ and the induced orderings are the same. The path $\upalpha \colonequals s_1s_1s_2s_1^{-1}$ starting at vertex $M_1$ is shown in the diagram on the right, where you travel along the first edge in the opposite orientation.
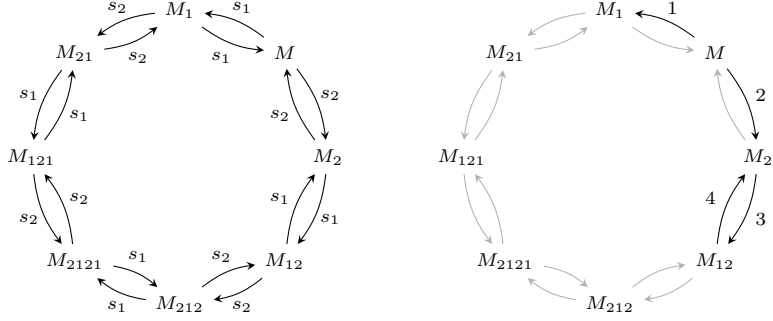
\[
\begin{tikzpicture}[scale=1.3,bend angle=15, looseness=1,>=stealth]
\node (C+) at (45:1.5cm) [] {$ \scriptstyle M$};
\node (C1) at (90:1.5cm) [] {$ \scriptstyle M_1$};
\node (C2) at (135:1.5cm) [] {$ \scriptstyle M_{21}$};
\node (C3) at (180:1.5cm) [] {$ \scriptstyle M_{121}$};
\node (C-) at (225:1.5cm) [] {$\scriptstyle M_{2121}$};
\node (C4) at (270:1.5cm) [] {$\scriptstyle M_{212}$};
\node (C5) at (315:1.5cm) [] {$\scriptstyle M_{12}$};
\node (C6) at (0:1.5cm) [] {$\scriptstyle M_{2}$};
\draw[->, bend right]  (C+) to (C1);
\draw[->, bend right]  (C1) to (C+);
\draw[->, bend right]  (C1) to (C2);
\draw[->, bend right]  (C2) to (C1);
\draw[->, bend right]  (C2) to (C3);
\draw[->, bend right]  (C3) to (C2);
\draw[->, bend right]  (C3) to (C-);
\draw[->, bend right]  (C-) to  (C3);
\draw[<-, bend right]  (C+) to  (C6);
\draw[<-, bend right]  (C6) to  (C+);
\draw[<-, bend right]  (C6) to  (C5);
\draw[<-, bend right]  (C5) to (C6);
\draw[<-, bend right]  (C5) to  (C4);
\draw[<-, bend right]  (C4) to (C5);
\draw[<-, bend right]  (C4) to  (C-);
\draw[<-, bend right]  (C-) to (C4);
\node at (67.5:1.1cm) {$\scriptstyle s_1$};
\node at (67.5:1.65cm) {$\scriptstyle s_1$};
\node at (112.5:1.1cm) {$\scriptstyle s_2$};
\node at (112.5:1.65cm) {$\scriptstyle s_2$};
\node at (157.5:1.1cm) {$\scriptstyle s_1$};
\node at (157.5:1.65cm) {$\scriptstyle s_1$};
\node at (202.5:1.1cm) {$\scriptstyle s_2$};
\node at (202.5:1.65cm) {$\scriptstyle s_2$};
\node at (247.5:1.1cm) {$\scriptstyle s_1$};
\node at (247.5:1.65cm) {$\scriptstyle s_1$};
\node at (292.5:1.1cm) {$\scriptstyle s_2$};
\node at (292.5:1.65cm) {$\scriptstyle s_2$};
\node at (337.5:1.1cm) {$\scriptstyle s_1$};
\node at (337.5:1.65cm) {$\scriptstyle s_1$};
\node at (382.5:1.1cm) {$\scriptstyle s_2$};
\node at (382.5:1.65cm) {$\scriptstyle s_2$};
\end{tikzpicture}
\hspace{1cm}
\begin{tikzpicture}[scale=1.3,bend angle=15, looseness=1,>=stealth]
\node (C+) at (45:1.5cm) [] {$ \scriptstyle M$};
\node (C1) at (90:1.5cm) [] {$ \scriptstyle M_1$};
\node (C2) at (135:1.5cm) [] {$ \scriptstyle M_{21}$};
\node (C3) at (180:1.5cm) [] {$ \scriptstyle M_{121}$};
\node (C-) at (225:1.5cm) [] {$\scriptstyle M_{2121}$};
\node (C4) at (270:1.5cm) [] {$\scriptstyle M_{212}$};
\node (C5) at (315:1.5cm) [] {$\scriptstyle M_{12}$};
\node (C6) at (0:1.5cm) [] {$\scriptstyle M_{2}$};
\draw[, ->, bend right]  (C+) to (C1);
\draw[black!30, ->, bend right]  (C1) to (C+);
\draw[black!30, ->, bend right]  (C1) to (C2);
\draw[black!30, ->, bend right]  (C2) to (C1);
\draw[black!30, ->, bend right]  (C2) to (C3);
\draw[black!30, ->, bend right]  (C3) to (C2);
\draw[black!30, ->, bend right]  (C3) to (C-);
\draw[black!30, ->, bend right]  (C-) to  (C3);
\draw[black!30, <-, bend right]  (C+) to  (C6);
\draw[ <-, bend right]  (C6) to  (C+);
\draw[ <-, bend right]  (C6) to  (C5);
\draw[ <-, bend right]  (C5) to (C6);
\draw[black!30, <-, bend right]  (C5) to  (C4);
\draw[black!30, <-, bend right]  (C4) to (C5);
\draw[black!30, <-, bend right]  (C4) to  (C-);
\draw[black!30, <-, bend right]  (C-) to (C4);
\node at (67.5:1.65cm) {$\scriptstyle 1$};
\node at (337.5:1.1cm) {$\scriptstyle 4$};
\node at (337.5:1.65cm) {$\scriptstyle 3$};
\node at (382.5:1.65cm) {$\scriptstyle 2$};
\end{tikzpicture}
\]
\end{example}

Not only will paths in $\umut(M)$ correspond to mutation sequences, but also to derived equivalences between the algebras of interest.  Combining Proposition \ref{derivedequiv} with Theorem \ref{rickards}, for each arrow $s_i \colon \overline{N} \to \overline{N}'$ in $\umut(M)$ there exists a derived equivalence
\begin{align}
F_i \colon \Db(\uEnd_\cE(N)) &\to \Db(\uEnd_\cE(N')) \label{deffunc} \\ 
\upmu_i (\uEnd_\cE(N)) &\mapsto \uEnd_\cE(N') \nonumber
\end{align}
where, giving $\uEnd_\cE(N)$ the induced ordering $\uHom_\cE(N,N_1), \dots, \uHom_\cE(N, N_n)$ from $\overline{N}$, the $j^{th}$ summand of $\upmu_i (\uEnd_\cE(N))$ maps to $\uHom_\cE(N', N'_j)$. Fix such a derived equivalence $F_i$ for each arrow $s_i$. For example, we could choose $F_i$ to be the unique (up to algebra isomorphism) standard derived equivalence associated to $\upmu_i (\uEnd_\cE(N))$, as discussed in \S\ref{derivedprelim}.
\begin{notation} \label{notation}
Consider a path $\upalpha \colonequals s_{i_m}^{\upvarepsilon_{m}} \dots s_{i_1}^{\upvarepsilon_{1}}$ in $\umut(M)$ starting at a vertex $\overline{N}$. Then, writing $\Gamma \colonequals \uEnd_\cE(N)$ set:
\begin{enumerate}
\item $\upnu_{\upalpha}N \colonequals \upnu_{i_m}^{\upvarepsilon_m} \dots \upnu_{i_1}^{\upvarepsilon_i}N$ and $\upmu_{\upalpha}\Gamma \colonequals \upmu^{\upvarepsilon_{m}}_{i_m}  \dots \upmu^{\upvarepsilon_{1}}_{i_1} \Gamma$.
\item $F_{\upalpha} \colonequals F^{\upvarepsilon_{m}}_{i_m} \circ \dots \circ F^{\upvarepsilon_{1}}_{i_1} \colon \Db(\Gamma) \to \Db(\uEnd_\cE(\upnu_\upalpha N))$.
\end{enumerate}
\end{notation}
\begin{example}
Consider the path $\upalpha \colonequals s_1s_1s_2s_1^{-1}$ from Example \ref{ex1}. Using that right and left mutation are equal for maximal rigid objects, the object at the end of the path is 
\begin{align*}
\upnu_\upalpha M_1 &\colonequals \upnu_1 \upnu_1 \upnu_2 \upnu_1^{-1} M_1\\ &\cong \upnu_1 \upnu_1 \upnu_2 \upnu_1^{-1} \upnu_1 M \\
&\cong \upnu_2 \upnu_1^{-1} \upnu_1 M\\
&\cong M_{2}
\end{align*}
as in the diagram. Similarly, but with no cancellation now as left and right mutation are different,
\vspace{-0.1cm}
\begin{align*}
\upmu_\upalpha \uEnd_\cE(M_1)  &\colonequals \upmu_1 \upmu_1 \upmu_2 \upmu_1^{-1} \uEnd_\cE(M_1).
\end{align*}
\end{example}
The following is the main technical result of this section. 
\begin{proposition} \label{mutationpaths}
Under the setup of \ref{setup2}, choose a rigid object $M \colonequals \bigoplus_{i=1}^n M_i$ in $ \ucE$ and let $\upalpha \colonequals s_{i_m}^{\upvarepsilon_{m}} \dots s_{i_1}^{\upvarepsilon_{1}}$ be a path in $\umut(M)$ starting at $\overline{N}$. Writing $\Lambda \colonequals \uEnd_\cE(N)$ and $\Gamma \colonequals \uEnd_\cE(\upnu_{\upalpha}N)$, the following hold.
\begin{enumerate}
\item $F_\upalpha( \upmu_\upalpha \Lambda) \cong \Gamma$ in $\Db(\Gamma)$.
\item $\End_{\Lambda}(\upmu_\upalpha \Lambda) \cong \Gamma$ as $k$-algebras.
\end{enumerate}
\end{proposition}
\begin{proof}
Note that $(2)$ follows easily from part $(1)$ since $F_\upalpha$ is an equivalence. We will prove part $(1)$ by induction on $m$. 

\textbf{Base Case $\mathbf{m=1}$:} If $\upalpha \colonequals s_i$, then the result follows from our choice of the $F_i$. If $\upalpha \colonequals s_i^{-1}$ for $s_i \colon  \upnu_i^{-1} N \to N$, then the assumption given on $F_i$ is that
\begin{align} \label{derived}
F_i( \upmu_i \uEnd_\cE(\upnu_i^{-1}N)) = \Lambda.
\end{align}
But then,
\begin{align*}
F_i^{-1}(\upmu_i^{-1} \Lambda) &\cong \upmu_i^{-1}F_i^{-1}( \Lambda) \tag{by Lemma \ref{tracking}}\\
&\cong \upmu_i^{-1}\left( \upmu_i \uEnd_\cE(\upnu_i^{-1}N) \right) \tag{by \eqref{derived}}\\
&\cong \uEnd_\cE(\upnu_i^{-1}N) ,
\end{align*}
as required.

\textbf{Inductive Step:} Let $\upbeta \colonequals s_{i_{m-1}}^{\upvarepsilon_{{m-1}}} \dots s^{\upvarepsilon_{1}}_{i_1} $ so by the inductive hypothesis, $F_\upbeta(\upmu_\upbeta (\Lambda)) \cong \uEnd_\cE(\upnu_\upbeta N)$. Then,
\begin{align*}
F_\upalpha \big(\upmu_\upalpha (\Lambda)\big) &\cong F^{\upvarepsilon_{m}}_{i_m}\big( F_\upbeta \big(\upmu_{i_{m}}^{\upvarepsilon_{m}} \upmu_\upbeta (\Lambda)\big) \big) \\
&\cong F^{\upvarepsilon_{m}}_{i_m}\big(\upmu_{i_m}^{\upvarepsilon_{m}} F_\upbeta \big(\upmu_\upbeta (\Lambda)\big) \big) \tag{by Lemma \ref{tracking}}\\
&\cong  F^{\upvarepsilon_{m}}_{i_m}\big(\upmu_{i_m}^{\upvarepsilon_{m}} \uEnd_\cE(\upnu_\upbeta N) \big) \tag{using inductive hypothesis}\\
&\cong \Gamma
\end{align*}
where the last isomorphism holds by applying the base case to the path $s_{i_m}^{\upvarepsilon_m}$ from $\upnu_\upbeta N$ to  $\upnu_\upalpha N$. 
\end{proof}
This has the following easy corollary.
\begin{cor}
Under the setup of \ref{setup2}, let $M \colonequals \bigoplus_{i=1}^n M_i$ be a rigid object of $\ucE$ and $\Lambda \colonequals \uEnd_\cE(M)$. Then, any tilting complex obtained from $\Lambda$ by finite iterated mutation (either left or right at each stage) has endomorphism algebra isomorphic to one of
\begin{align*}
\{ \uEnd_\cE(N) \mid N \in \mut(M) \}.
\end{align*}
\end{cor}
\begin{proof}
If $T$ is a tilting complex for $\Lambda$ obtained by iterated mutation then 
\begin{align*}
T \cong  \upmu^{\upvarepsilon_{m}}_{i_m}  \dots \upmu^{\upvarepsilon_{1}}_{i_1} \Lambda
\end{align*}
for some $i_1, \dots, i_m \in \{1, \dots n\}$ and $\upvarepsilon_{i} \in \{-1,1\}$. This defines a path $\upalpha \colonequals s^{\upvarepsilon_{m}}_{i_m}  \dots s^{\upvarepsilon_{1}}_{i_1}$ in $\umut(M)$ starting at $M$ for which $T \cong \upmu_\upalpha \Lambda$. Then by Proposition \ref{mutationpaths}
\begin{align*}
\hspace{3.9cm} \End_\Lambda (T) \cong \End_\Lambda( \upmu_\upalpha \Lambda)  \cong \uEnd_{\cE}( \upnu_\upalpha M). \hspace{3.5cm} \qedhere
\end{align*}
\end{proof}

Using this result, to completely determine the basic algebras in the derived equivalence class, we just need to show that every basic tilting complex for such a $\Lambda$ can be obtained by iterated right or left mutation from $\Lambda$. By Lemma \ref{tiltingconnected}, this will follow from showing $\Lambda$ is silting-discrete.

\begin{thm} \label{contractiltingdis}
Under the setup of \ref{setup2}, let $M \in \ucE$ be a rigid object and $\Lambda \colonequals \uEnd_\cE(M)$. Then $\Lambda$ is a silting-discrete algebra. 
\end{thm}
\begin{proof}
We will check condition $(3)$ of Theorem \ref{tiltingdiscrete}, namely that $ \twosilt_P \Lambda$ is a finite set for any silting object $P$ which is given by iterated left mutation from $\Lambda$.

By Theorem \ref{ctbij}, two-term silting complexes for $\Lambda$ are in bijection with certain rigid objects of $\ucE$ lying in $M *\Sigma M$. However, recall from Remark \ref{finiterigid}(2) that our setup ensures that there are only finitely many rigid objects in $\ucE$ and hence in $M *\Sigma M$. Thus, $\twosilt_{\Lambda} \Lambda= \twosilt \Lambda$ is finite.

Now suppose $T$ is a silting (and hence in this case tilting by Proposition \ref{symmalg}) complex obtained by iterated left mutation of $\Lambda$. Thus, $T \cong \upmu_\upalpha \Lambda$ for some positive path $\upalpha$ in $\umut(M)$ starting at $M$. If $\upalpha$ ends at the rigid object $N$, then writing $\Gamma \colonequals \uEnd_\cE(N)$, the associated equivalence
 \begin{align*}
F_\upalpha \colon \Db(\Lambda) &\to \Db(\Gamma)
\end{align*}
maps $T$ to $\Gamma$ by Proposition \ref{mutationpaths}. Thus, applying Lemma \ref{tracking} to $F_\upalpha$, there is a bijection
\begin{align*}
 \{P \in \silt \Lambda \mid T \geq P \geq T[1] \} \longleftrightarrow \{Q \in \silt \Gamma \mid \Gamma \geq Q \geq \Gamma[1] \}.
\end{align*}
By definition, the left hand side is $\twosilt_T \Lambda$, while the right hand side is $\twosilt \Gamma$.
But since $\Gamma$ is also the endomorphism algebra of a rigid object in $\ucE$, the first argument shows $\twosilt \Gamma$ must be finite and hence so is $\twosilt_T \Lambda$.
\end{proof}
The following is our main result. Part $(1)$ is a consequence of Theorem \ref{contractiltingdis} showing how to view all tilting complexes combinatorially in $\umut(M)$ and part $(2)$ is a consequence of the fact this viewpoint allows us to control the endomorphism rings of the tilting complexes via Proposition \ref{mutationpaths}.
\begin{cor} \label{mainresult}
In the setup of \ref{setup2}, choose a rigid object $M \in \ucE$ and write $ \Lambda \colonequals \uEnd_\cE(M)$. Then the following statements hold.
\begin{enumerate} 
\item Any tilting complex of $\Lambda$ is isomorphic to $\upmu_\upalpha \Lambda$ (defined in \ref{notation}) for some (not necessarily positive) path $\upalpha$ in $\umut(M)$ starting at vertex $M$.
\item The basic algebras derived equivalent to $\Lambda$ are precisely the algebras 
\begin{align*}
\{ \uEnd_\cE(N) \mid N \in \mut(M) \},
\end{align*}
of which there are only finitely many.
\end{enumerate}
\end{cor}
\begin{proof}
\begin{enumerate}[leftmargin=0cm,itemindent=.6cm,labelwidth=\itemindent,labelsep=0cm,align=left]
\item Take $T \in \tilt \Lambda$. Since $\Lambda$ is silting-discrete by Theorem \ref{contractiltingdis}, Proposition \ref{tiltingconnected} shows
\begin{align*}
T \cong  \upmu^{\upvarepsilon_{m}}_{i_m}  \dots \upmu^{\upvarepsilon_{1}}_{i_1} \Lambda
\end{align*}
for some $i_1, \dots, i_m \in \{1, \dots n\}$ and $\upvarepsilon_{i} \in \{-1,1\}$. This defines a path $\upalpha \colonequals s^{\upvarepsilon_{m}}_{i_m}  \dots s^{\upvarepsilon_{1}}_{i_1}$ starting at $M$ for which $T \cong \upmu_\upalpha \Lambda$ by definition.
\item Suppose $\Gamma $ is a basic algebra derived equivalent to $\Lambda$. Then $\Gamma \cong \End_\Lambda(T)$ for some basic tilting complex $T \in \tilt \Lambda$ by Theorem \ref{rickards}. However, by part $(1)$ this shows $\Gamma \cong  \End_\Lambda(\upmu_\upalpha \Lambda)$ for some path $\upalpha$ starting at $M$ and hence by Proposition \ref{mutationpaths}, $\Gamma \cong  \uEnd_\cE(\upnu_\upalpha M)$ as required. Finally, the set is finite since, by Remark \ref{finiterigid}(2), our setup ensures that there are only finitely many basic rigid objects in $\ucE$, and hence in $\mut(M)$. \qedhere
\end{enumerate}
\end{proof}
Recall from Remark \ref{maxrigremark}(3) that, as there are only finitely many basic maximal rigid objects in $\ucE$, $\mut (M) = \mrig \ucE$ for any maximal rigid object $M \in \ucE$. Therefore, when restricted to maximal rigid objects, part $(2)$ of Corollary \ref{mainresult} specialises to the following.
\begin{cor} \label{maxrig}
In the setup of \ref{setup2}, choose a maximal rigid object $M \in \ucE$ and write $ \Lambda \colonequals \uEnd_\cE(M)$. Then the basic algebras derived equivalent to $\uEnd_\cE(M)$ are precisely the stable endomorphism algebras of maximal rigid objects in $\ucE$. In particular, there are only finitely many such algebras.
\end{cor}

\begin{example} \label{ex2}
Returning to Example \ref{ex1}, we set $\Lambda \colonequals \uEnd_\cE(M)$ and $\Lambda_{i_m \dots i_1} \colonequals \uEnd_\cE(M_{i_m \dots i_1})$. Corollary \ref{mainresult} shows that the diagram
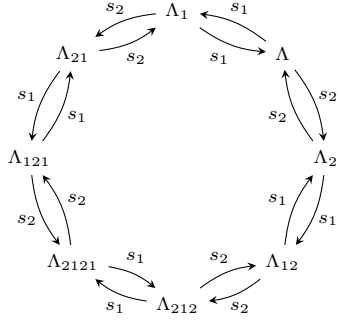
\begin{center}
 \begin{tikzpicture}[scale=1.3,bend angle=15, looseness=1,>=stealth]
\node (C+) at (45:1.5cm) [] {$ \scriptstyle \Lambda$};
\node (C1) at (90:1.5cm) [] {$ \scriptstyle \Lambda_1$};
\node (C2) at (135:1.5cm) [] {$ \scriptstyle \Lambda_{21}$};
\node (C3) at (180:1.5cm) [] {$ \scriptstyle \Lambda_{121}$};
\node (C-) at (225:1.5cm) [] {$\scriptstyle \Lambda_{2121}$};
\node (C4) at (270:1.5cm) [] {$\scriptstyle \Lambda_{212}$};
\node (C5) at (315:1.5cm) [] {$\scriptstyle \Lambda_{12}$};
\node (C6) at (0:1.5cm) [] {$\scriptstyle \Lambda_{2}$};
\draw[->, bend right]  (C+) to (C1);
\draw[->, bend right]  (C1) to (C+);
\draw[->, bend right]  (C1) to (C2);
\draw[->, bend right]  (C2) to (C1);
\draw[->, bend right]  (C2) to (C3);
\draw[->, bend right]  (C3) to (C2);
\draw[->, bend right]  (C3) to (C-);
\draw[->, bend right]  (C-) to  (C3);
\draw[->, bend right]  (C+) to  (C6);
\draw[->, bend right]  (C6) to  (C+);
\draw[->, bend right]  (C6) to  (C5);
\draw[->, bend right]  (C5) to (C6);
\draw[->, bend right]  (C5) to  (C4);
\draw[->, bend right]  (C4) to (C5);
\draw[->, bend right]  (C4) to  (C-);
\draw[->, bend right]  (C-) to (C4);
\node at (67.5:1.1cm) {$\scriptstyle s_1$};
\node at (67.5:1.65cm) {$\scriptstyle s_1$};
\node at (112.5:1.1cm) {$\scriptstyle s_2$};
\node at (112.5:1.65cm) {$\scriptstyle s_2$};
\node at (157.5:1.1cm) {$\scriptstyle s_1$};
\node at (157.5:1.65cm) {$\scriptstyle s_1$};
\node at (202.5:1.1cm) {$\scriptstyle s_2$};
\node at (202.5:1.65cm) {$\scriptstyle s_2$};
\node at (247.5:1.1cm) {$\scriptstyle s_1$};
\node at (247.5:1.65cm) {$\scriptstyle s_1$};
\node at (292.5:1.1cm) {$\scriptstyle s_2$};
\node at (292.5:1.65cm) {$\scriptstyle s_2$};
\node at (337.5:1.1cm) {$\scriptstyle s_1$};
\node at (337.5:1.65cm) {$\scriptstyle s_1$};
\node at (382.5:1.1cm) {$\scriptstyle s_2$};
\node at (382.5:1.65cm) {$\scriptstyle s_2$};
\end{tikzpicture}
\end{center}
can be thought of as a `picture' of the derived equivalence class of $\Lambda$. The vertices are precisely the basic algebras in the derived equivalence class and the (not necessarily positive) paths starting at a given vertex control all of the tilting complexes of that algebra. Further, the end vertex of a path determines the endomorphism algebra of the associated tilting complex and thus we think of each path as a derived equivalence induced by that tilting complex. 
If we recall from \S \ref{derivedprelim} that any tilting complex induces a unique standard derived equivalence (up to algebra isomorphism), we obtain the following result.
\end{example}

\begin{cor} \label{standard}
Under the setup of \ref{setup2}, choose a rigid object $M \in \ucE$ and write $\Lambda  \colonequals \uEnd_\cE(M)$. If the $F_i$ of \eqref{deffunc} are chosen to be standard equivalences, then the following statements hold.
\begin{enumerate}
\item For any path $\upalpha \colon \overline{N} \to \overline{N}' $ in $\umut(M)$ the derived equivalence $F_\upalpha \colon \Db(\uEnd_\cE(N)) \to \Db(\uEnd_\cE(N'))$ given in \ref{notation} is a standard equivalence mapping the tilting complex $\upmu_\upalpha \uEnd_\cE(N)$ to $\uEnd_\cE(N')$.
\item Up to algebra isomorphism, any standard equivalence from $\Db(\Lambda)$ is obtained by composition of the $F_i$ and their inverses.
\end{enumerate}
\end{cor}
\begin{proof}
\begin{enumerate}[leftmargin=0cm,itemindent=.6cm,labelwidth=\itemindent,labelsep=0cm,align=left]
\item Take a path $\upalpha \colon \overline{N} \to \overline{N}'$ in $\umut(M)$. If $\Gamma \colonequals \uEnd_\cE(N)$ and $\Delta \colonequals \uEnd_\cE(N')$, then as in \ref{notation}, $\upalpha$ corresponds to a tilting complex $\upmu_\upalpha \Gamma$ and a derived equivalence
\begin{align*}
 F_{\upalpha} \colon \Db(\Gamma) &\to \Db(\Delta)\\
\upmu_\upalpha \Gamma &\mapsto \Delta 
\end{align*}
using Proposition \ref{mutationpaths}. If the $F_i$ are standard equivalences, then $F_\upalpha$ must also be and hence $F_\upalpha$ is the unique (up to algebra isomorphism) standard equivalence associated to the tilting complex $\upmu_\upalpha \Gamma$.
\item Take any standard equivalence $F \colon \Db(\Lambda) \to \Db(\Gamma)$. Then $F^{-1}(\Gamma)$ is a tilting complex for $\Lambda$ and so by Propositions \ref{tiltingconnected} and  \ref{contractiltingdis}, $F^{-1}(\Gamma) \cong \upmu_\upalpha \Lambda$ for some path $\upalpha$ in $\umut(M)$ starting at $M$. By part $(1)$, $F_\upalpha$ is also a standard equivalence induced by $\upmu_\upalpha \Lambda$ and hence $F$ and $F_\upalpha$ must be the same up to an algebra isomorphism by Proposition \ref{rouq}.  \qedhere
\end{enumerate}
\end{proof}
In this way, we can say the diagram in Example \ref{ex2} not only contains all the basic members of the derived equivalence class but also all the standard equivalences between them (up to algebra isomorphism) and that these correspond precisely to the paths. In this way, we have a complete picture of the derived equivalence class of these algebras. 
\begin{remark}
In this setting, there is no way of telling when two paths give rise to the same derived equivalence. To be able to say this, we would need to define explicit standard equivalences and then understand how to compose them. This idea is explored in \cite{hyperplane} where additional structure coming from the geometric setting is used in order to choose the $F_i$ explicitly. 
\end{remark}

\section{Geometric Application} \label{geometricapplication}
In this section, we apply the results from the previous section to certain invariants of $3$-fold flopping contractions, known as contraction algebras. 

\subsection{Complete Local Setup}
For our purposes, a 3-fold flopping contraction is a projective birational morphism $f \colon X \to X_\con$ between Gorenstein normal $\mathbb{C}$-schemes of dimension three satisfying $\Rf_*\mathcal{O}_X = \mathcal{O}_{X_\con}$, and which further is an isomorphism in codimension one. When the base of the flopping contraction is affine and complete local, contraction algebras have a very explicit construction and so we begin with this case. 
\begin{setup} \label{comploc}
Take $f \colon X \to \Spec \rR$ to be a $3$-fold flopping contraction where $\rR$ is complete local and $X$ has at worst Gorenstein terminal singularities.
\end{setup}
The condition on $X$ in this setup is satisfied when $X$ is smooth but also allows for some mild singularities. Moreover, the condition also forces $\Spec \rR$ to have at worst Gorenstein terminal singularities and hence to be an isolated cDV singularity \cite{reid}. 
\begin{definition}
A three dimensional complete local $\mathbb{C}$-algebra $\rR$ is a \emph{compound Du Val (cDV) singularity} if $\rR$ is isomorphic to 
\begin{align*}
\mathbb{C}\llbracket u,v,x,y\rrbracket/ (f(u,v,x) + yg(u,v,x,y)) 
\end{align*}
where $\mathbb{C}\llbracket u,v,x \rrbracket/(f(u,v,x))$ is a Du Val surface singularity and $g$ is arbitrary. 
\end{definition}
We will sometimes wish to assume further that $X$ is $\mathbb{Q}$-factorial (see \cite[\S2]{HomMMP} for a definition) but will only do so if explicitly stated.
\begin{definition}
A $3$-fold flopping contraction $f \colon X \to \Spec \rR$ where $\rR$ is complete local and $X$ has at worst $\mathbb{Q}$-factorial terminal singularities is called a \emph{minimal model} of $\Spec \rR$.
\end{definition}
For a cDV singularity, it is well known that there are only finitely many minimal models \cite{[KM]} and one goal of the Homological Minimal Model Programme was to provide an algorithm that can produce all the minimal models from a given one, similar to how all maximal rigid objects can be obtained via iterated mutation from a given maximal rigid object in our previous setting. The key observation is that the maximal Cohen--Macaulay modules of $\rR$ link the two settings.
\begin{definition} 
Let $(R, \mathfrak{m})$ be a commutative noetherian local ring and choose $M \in \mod R$. Then define the depth of $M$ to be
\begin{align*}
\mathrm{depth}_R(M)=\mathrm{min}\{ i \geq 0 \mid \Ext^i_R(R/ \mathfrak{m},M) \neq 0\}.
\end{align*}
We say that $M$ is \emph{maximal Cohen--Macaulay} $(\mathrm{CM})$ if $\mathrm{depth}_R(M)=\mathrm{dim}(R)$ and write $\cmr$ for the full subcategory of $\mod R$ consisting of maximal Cohen--Macaulay modules.
\end{definition}
The following summary theorem asserts that $\cE \colonequals \cmrr$ satisfies all but the last condition of Setup \ref{setup2}. For details, and full references, see e.g. \cite[\S1]{symmetric}.
\begin{proposition}   \label{prop} 
If $\rR$ is a complete local isolated cDV singularity, $\cmrr$ is a Frobenius category. Moreover, the stable category $\ucmrr$ is a Krull-Schmidt, Hom-finite, 2-Calabi-Yau triangulated category with shift functor $\Sigma$ satisfying $\Sigma^2 \cong \id$. 
\end{proposition}

In particular, rigid objects in $\ucmrr$ can be defined as in Definition \ref{definerigid}. Moreover, in Theorem \ref{HMMPbij}, we will see that $\ucmrr$ also satisfies the last condition of setup \ref{setup2} and thus we will be able to apply our main results to this setting. However, for this to be useful in studying the geometry, the rigid objects need to have connection with the geometry. This link comes about via the contraction algebras; an invariant of $3$-fold flops introduced by Donovan--Wemyss \cite{DefFlops, twists}. 

\subsection{Construction of the Contraction Algebra} \label{construction}
In the setting of \ref{comploc}, the contraction algebra attached to $f$ has a very explicit construction, detailed in \cite[3.5]{ConDef}, but provided here for convenience.

It is well known that in the setup of \ref{comploc}, $\Spec \rR$ has a unique singular point $\mathfrak{m}$ and the preimage $C \colonequals f^{-1}(\mathfrak{m})$ consists of a chain of curves. In particular, giving $C$ the reduced scheme structure, we have $C^{\mathrm{red}} = \mathop{\cup}\limits_{i=1}^n C_i$ where $C_i \cong \mathbb{P}^1$. 

For each $i$, let $\mathcal{L}_i$ be the line bundle on $X$ such that $\mathcal{L}_i \cdot C_j = \updelta_{ij}$. If the multiplicity of $C_i$ is equal to $1$, set $\mathcal{M}_i \colonequals \mathcal{L}_i$. Otherwise, define $\mathcal{M}_i$ to be given by the maximal extension 
 \begin{align*}
0 \to \mathcal{O}_X^{\oplus (r_i-1)} \to \mathcal{M}_i \to \mathcal{L}_i \to 0
\end{align*}
associated to a minimal set of $r_i-1$ generators of $H^1(X, \mathcal{L}_i^*)$ as an $\rR$-module \cite[3.5.4]{VdB}. Then, by \cite[3.5.5]{VdB},
\begin{align*}
\mathcal{O}_X \oplus \bigoplus_{i=1}^n \mathcal{M}_i^*
\end{align*}
is a tilting bundle on $X$. Associated to this is the algebra $A \colonequals \End_X(\mathcal{O}_X \oplus \bigoplus\limits_{i=1}^n \mathcal{M}_i^*)$ which is derived equivalent to the category of coherent sheaves on $X$. Pushing forward via $f$ gives $f_*(\mathcal{O}_X) \cong \rR$ and, for each $i$, $f_*(\mathcal{M}_i^*) \cong N_i$ for some $\rR$-module $N_i$. Since $f$ is a flopping contraction, there is an isomorphism \cite[3.2.10]{VdB}
\begin{align*}
A \cong \End_\rR(\rR \oplus \bigoplus_{i=1}^n N_i).
\end{align*} 

\begin{definition}
The \emph{contraction algebra} associated to $f$ is defined to be 
\begin{align*}
A_\con \colonequals \End_\rR(\rR \oplus \bigoplus_{i=1}^n N_i)/[\rR]\cong \uEnd_\rR(\rR \oplus \bigoplus_{i=1}^n N_i)
\end{align*}
where $[\rR]$ denotes the ideal of all morphisms which factor through $\add{\rR}$.
\end{definition}
\begin{remark}
The contraction algebra $A_\con$ can also be defined as the representing object of a certain deformation functor of the curves in $C$, but we will not need this here.
\end{remark}

\subsection{Summary of Results from the Homological Minimal Model Programme}
The following key proposition provides the link between contraction algebras and cluster-tilting theory.
\begin{proposition} \label{rigidcon}
Assuming the setup of \ref{comploc} and with notation as above, the following hold.
\begin{enumerate}
\item Each $N_i$ lies in $\cmrr$ and further, $N\colonequals \bigoplus\limits_{i=1}^nN_i$ is a rigid object in $\ucmrr$. 
\item $f$ is a minimal model if and only if $N \colonequals \bigoplus\limits_{i=1}^n N_i$ is maximal rigid in $\ucmrr$.
\end{enumerate}
\end{proposition}
\begin{proof}
\begin{enumerate}[leftmargin=0cm,itemindent=.6cm,labelwidth=\itemindent,labelsep=0cm,align=left]

\item As $f \colon X \to \Spec \rR$ is crepant, \cite[4.14]{[IW2]} shows that $\End_\rR(\rR \oplus N)$ lies in $\cmrr$ and hence, as $N_i \cong \Hom_\rR(\rR, N_i)$ is a direct summand, $N_i$ also lies in $\cmrr$. Thus, $N$ is a \textit{modifying module} in the terminology of \cite[4.1]{mmas} and hence, by \cite[5.12]{mmas}, satisfies $\Ext^1_\rR(N,N) =0$ which is equivalent to $N$ being rigid in $\ucmrr$.
\item By \cite[4.16]{[IW2]}, $f$ is a minimal model if and only if $\rR \oplus N$ is a \textit{maximal modifying module} and thus if and only if $N$ is a maximal rigid object in $\ucmrr$ by \cite[5.12]{mmas}. \qedhere
\end{enumerate}
\end{proof}

This proposition shows that the construction in \S \ref{construction} actually gives a well-defined map
\begin{align}
\{\text{flopping contractions as in setup \ref{comploc}} \} \to \rig \ucmrr \label{map1}
\end{align}
which restricts to a map
\begin{align}
\{\text{minimal models of $\Spec \rR$} \} \to \mrig \ucmrr. \label{map2}
\end{align}
Since $N \colonequals \bigoplus\limits_{i=1}^n N_i$ is a rigid object in $\ucmrr$, we can choose any summand $N_i$ to mutate at and produce a new rigid object $\upnu_i N$. Alternatively, consider the curve $C_i$ in the exceptional locus corresponding to the summand $N_i$. Since $f \colon X \to \Spec \rR$ is a flopping contraction and $\rR$ is complete local, choosing any such $C_i$, it is possible to factorise $f$ as
\begin{align*}
X \xrightarrow{g} X_\con \xrightarrow{h} \Spec \rR
\end{align*}
where $g(C_j)$ is a single point if and only if $j=i$. For any such factorisation, there exists a certain birational map $g^+ \colon X^+ \to X_\con$, satisfying some technical conditions detailed in \cite[2.6]{HomMMP}, which fits into a commutative diagram
\begin{center}
\begin{tikzpicture}
\node (C1) at (0, 0) {\scriptsize{$X_\con$}};
\node (C2) at (0, -1.2) {\scriptsize{$\Spec \rR$}};
\node (C3) at (-1.2, 1.2) {\scriptsize{$X$}};
\node (C4) at (1.2, 1.2)    {\scriptsize{$X^+$}};
\node at (0.55, 0.71)    {\scriptsize{$g^+$}};
\node at (-0.55, 0.71)    {\scriptsize{$g$}};
\node at (-0.1, -0.5)    {\scriptsize{$h$}};
\draw[->] (C1) -- (C2); 
\draw[->] (C3) to node[left] {$\scriptstyle f$} (C2);
\draw[->] (C4) to node[right] {$\scriptstyle f^+$} (C2);
\draw[->] (C3) -- (C1);
\draw[->] (C4) -- (C1);
\draw[->,dashed] (C3) to node[above] {$\scriptstyle \phi$} (C4);
\end{tikzpicture}
\end{center} 
where $\phi$ is a birational equivalence (see e.g. \cite[p25]{kollar} or \cite[\S 2]{Schroer}). We call $f^+ \colon X^+ \to \Spec \rR$ the \emph{simple flop} of $f$ at the curve $ C_i$. Since $f^+$ is again a flopping contraction, it also has an associated contraction algebra, constructed as in \S\ref{construction}. 

We can therefore consider two algebras obtained by `mutation' from $f$:
\begin{enumerate}
\item The algebra $\uEnd_R(\upnu_i N)$ obtained by mutating $N$ at summand $N_i$.
\item The contraction algebra $\uEnd_R(M)$ associated to $f^+$\hspace{-1.5pt}, where $f^+$ is the simple flop of $f$ at the curve $C_i$.
\end{enumerate}
The following proposition shows that these algebras are isomorphic.
\begin{proposition} \cite[4.20(1)]{HomMMP} \label{flopsmut}
Under the setup of \ref{comploc}, suppose that the contraction algebra associated to $f$ is $\uEnd_\rR(N)$. If $f^+$ is the flop of $f$ at $C_i$, then the contraction algebra of $f^+$ is $\uEnd_\rR(\upnu_i N)$.
\end{proposition}
Since $f^+$ will have the same number of curves in the exceptional locus as $f$, we can consider iterated flops.
\begin{notation}
Let $f$ be a flopping contraction as in setup \ref{comploc} and suppose there are curves $C_1, \dots, C_n$ in the exceptional locus of $f$. Given a sequence $(i_1, \dots, i_m)$ where $i_j \in \{1, \dots n \}$ we obtain a flopping contraction $f_{i_m \dots i_1}$ defined iteratively via:
\begin{enumerate}
\item $f_{i_1}$ is the simple flop of the $f$ at curve $C_{i_1}$.
\item $f_{i_j \dots i_1}$ is the flop of $f_{i_{j-1}\dots i_1}$ at the curve $C_{i_j}$ for $1 < j \leq m$.
\end{enumerate} 
We call the sequence $(i_1, \dots, i_m)$ a \emph{mutation sequence}.
\end{notation}
Repeated use of Proposition \ref{flopsmut} shows that the contraction algebra of $f_{i_m \dots i_1}$ is given by 
\begin{align*}
\uEnd_\rR(\upnu_{i_m} \dots \upnu_{i_1} N).
\end{align*}
In other words, the maps \eqref{map1} and \eqref{map2} respect mutation. Using this, and the fact there are only finitely many minimal models, it is possible to show that the map in \eqref{map2} is surjective \cite[4.10]{HomMMP}. Further, both maps can be seen to be injective \cite[4.4]{HomMMP} which gives the following result.

\begin{thm} \cite[4.10]{HomMMP} \label{HMMPbij}
If $\Spec \rR$ is a complete local isolated cDV singularity, then the construction in \S\ref{construction} yields a bijection 
\begin{align*}
\{\text{minimal models of $\Spec \rR$}\}  \longleftrightarrow  \mrig \ucmrr
\end{align*}
which respects mutation. Here, two minimal models are identified if they are isomorphic as $R$-schemes.
\end{thm}

Since $\Spec \rR$ has only finitely many minimal models, a direct consequence of Theorem \ref{HMMPbij} is that $\ucmrr$ must have only finitely many maximal rigid objects. 
\subsection{New Results}
Combining Proposition \ref{prop} and the remark after Theorem \ref{HMMPbij} shows that $\cmrr$ and its stable category satisfy the conditions in \ref{setup2}. Thus we can apply the results of the previous section to obtain the following, which is our main geometric result.
\begin{thm} \label{geomres}
Let $f \colon X \to \Spec \rR$ be as in setup \ref{comploc} with associated contraction algebra $A_\con \colonequals \uEnd_\rR(N)$. 
\begin{enumerate}
\item $A_\con$ is a silting-discrete algebra.
\item The endomorphism algebra of the tilting complex $T \colonequals \upmu^{\upvarepsilon_{m}}_{i_m}  \dots \upmu^{\upvarepsilon_{1}}_{i_1} A_\con$  is isomorphic to the contraction algebra of $f_{i_m\dots i_1}$.
\item The basic algebras derived equivalent to $A_\con$ are precisely the contraction algebras of flopping contractions $g \colon Y \to \Spec \rR$, obtained by a sequence of iterated flops from $f$. In particular, there are only finitely many such algebras.
\end{enumerate}
\end{thm}
\begin{proof}
 Recall that $\Spec \rR$ is a complete local isolated cDV singularity and thus, by Proposition \ref{prop} and the comments after Theorem \ref{HMMPbij}, $\cmrr$ and its stable category satisfy the conditions in \ref{setup2}.
\begin{enumerate}[leftmargin=0cm,itemindent=.6cm,labelwidth=\itemindent,labelsep=0cm,align=left]
\item As $N$ is rigid in $\ucmrr$ by Proposition \ref{rigidcon}, the algebra
\begin{align*}
A_\con \colonequals \uEnd_\rR(N)
\end{align*}
is silting-discrete by Theorem \ref{contractiltingdis}.
\item By Proposition \ref{mutationpaths}, there is an isomorphism 
\begin{align*}
\End_{A_{\con}}(T) \cong \uEnd_\rR(\upnu^{\upvarepsilon_{m}}_{i_m}  \dots \upnu^{\upvarepsilon_{1}}_{i_1} N).
\end{align*}
However, in this setting right and left mutation are equal by \cite[2.25]{HomMMP} and thus 
\begin{align*}
\End_{A_{\con}}(T) \cong \uEnd_\rR(\upnu_{i_m}  \dots \upnu_{i_1} N),
\end{align*}
which is the contraction algebra of $f_{i_m \dots i_1}$, by repeated use of Proposition \ref{flopsmut}.
\item Combining part $(1)$, Proposition \ref{tiltingconnected} and part $(2)$ shows that the endomorphism algebra of any basic tilting complex for $A_\con$ is isomorphic to the contraction algebra of some flopping contraction obtained from $f$ by a sequence of iterated flops. Applying Theorem \ref{rickards} gives the result.
\qedhere
\end{enumerate}
\end{proof}
In the special case of minimal models, it is well known that any two minimal models are connected by a sequence of simple flops \cite{[K]}. Thus, part $(3)$ of Theorem \ref{geomres} reduces to the following.
\begin{cor} \label{geommaxrigid}
Let $f \colon X \to \Spec \rR$ be a minimal model of a complete local isolated cDV singularity. Writing $A_\con$ for the associated contraction algebra, the basic algebras derived equivalent to $A_\con$ are precisely the contraction algebras of minimal models of $\Spec \rR$.  
\end{cor}
Recall from Conjecture \ref{conjwemyss}, that it is expected that the derived category of the contraction algebras of $\Spec R$ completely controls the geometry. The following shows that, to some extent, this is true.

\begin{thm} \label{evidence}
Suppose that $f \colon X \to \Spec \rR$ and $g \colon Y \to \Spec \sS$ are minimal models of complete local isolated cDV singularities with associated contraction algebras $A_\con$ and $B_\con$. If $A_\con$ and $B_\con$ are derived equivalent then there is a bijection 
\begin{align*}
\{\text{minimal models of $\Spec \rR$}\} \longleftrightarrow \{\text{minimal models of $\Spec \sS$}\}.
\end{align*}
Further, the bijection preserves both mutation and contraction algebras.
\end{thm}
\begin{proof}
Let $M \in \ucmrr$ be the maximal rigid object associated to $f$ and let $N \in \ucmss$ be the maximal rigid object associated to $g$ so that
\begin{align*}
A_\con \colonequals \uEnd_\rR(M) \quad \text{and} \quad B_\con \colonequals \uEnd_\sS(N).
\end{align*}
Fix an ordering $C_1, \dots, C_n$ of the curves in the exceptional locus of $f$, which fixes an ordering on the decomposition $M \cong \bigoplus_{i=1}^nM_i$ such that $M_i$ corresponds to curve $C_i$ via the construction in \S\ref{construction}.

Since $B_\con$ is basic and derived equivalent to $A_\con$, Corollary \ref{geommaxrigid} shows that $B_\con$ must be isomorphic to a contraction algebra for some minimal model of $\Spec \rR$. In particular, as any two minimal models are connected by a sequence of flops $B_\con$ must be the contraction algebra of
\begin{align*}
f' \colonequals f_{i_m \dots i_1} \colon X' \to \Spec \rR
\end{align*}
for some mutation sequence $(i_1, \dots, i_m)$ . Thus, writing $M' \colonequals \upnu_{i_m} \dots \upnu_{i_1} M$,
\begin{align*}
B_\con \cong \uEnd_\rR(M')
\end{align*}
by repeated use of Proposition \ref{flopsmut}. In particular, there exists a decomposition $N \cong \bigoplus_{i=1}^n N_i$  such that, for all $i=1, \dots, n$,
\begin{align*}
\uHom_\rR(M',M'_i) \cong \uHom_\sS(N,N_i)
\end{align*} 
as projective $B_\con$-modules. This also fixes a labelling $D_1, \dots, D_n$ of the curves in the exceptional locus of $g$  so that $N_i$ corresponds to $D_i$ via the construction in \S\ref{construction}.

Applying Theorem \ref{ctbij} and the remark afterwards, there are mutation preserving bijections 
\begin{align*}
\twosilt B_\con &\longleftrightarrow \mrig \ucmrr \hspace{-1.5cm}  &\text{and} \quad \twosilt B_\con &\longleftrightarrow \mrig \ucmss .\\
 B_\con &\mapsto M' &\  B_\con &\mapsto N 
\end{align*} 
Further, Theorem \ref{HMMPbij} shows there are bijections
\begin{align*}
\{\text{minimal models of $\Spec \rR$}\} &\longleftrightarrow   \mrig \ucmrr\\
f' &\mapsto M'
\end{align*}
and 
\begin{align*}
\{\text{minimal models of $\Spec \sS$}\}  &\longleftrightarrow   \mrig \ucmss \\
g &\mapsto N
\end{align*}
which both respect mutation. Combining all of these provides a mutation preserving bijection 
\begin{align*}
\{\text{minimal models of $\Spec \rR$}\} &\longleftrightarrow \{\text{minimal models of $\Spec \sS$}\}\\
f' &\mapsto g
\end{align*}
as required. Thus for any mutation sequence $(j_1, \dots, j_l)$, our choice of indexing on the curves ensures $f'_{j_l \dots j_1} \mapsto g_{j_l \dots j_1}$. Since the contraction algebra of both $f'$ and $g$ is $B_\con$, part $(2)$ of Theorem \ref{geomres} shows the contraction algebra of both $f'_{j_l \dots j_1}$ and $g_{j_l \dots j_1}$ is 
\begin{align*}
\End_{B_\con}(\upmu_{j_l}  \dots \upmu_{j_1} B_\con)
\end{align*}
and hence the bijection preserves contraction algebras.
\end{proof}
\subsection{Example} \label{geometricexample}
A large class of examples can be obtained from complete local $cA_{m-1}$ singularities. These can be written in the form
\begin{align*}
\rR \colonequals \mathbb{C}\llbracket u,v,x,y \rrbracket/ (uv-f(x,y))
\end{align*}
where $m$ is the order of the polynomial $f(x,y)$ considered as a power series. We will
only consider isolated singularities, which can be characterised using the irreducible
factors of $f(x, y)$; namely, if $f$ factors into $n$ irreducible power series $f_1, \dots, f_n$, then $\rR$
is isolated precisely when $(f_i) \neq (f_j)$ for all $i \neq j$. For these singularities, the maximal rigid objects in $\ucmrr$ have been completely determined.

\begin{definition}
Suppose $\mathbb{C}\llbracket u,x,x,v \rrbracket/ (uv-f_1 \dots f_{n})$ is a $cA_{m-1}$ singularity. Given $\sigma \in S_n$, where $S_n$ is the symmetric group on $n$ objects, define
\begin{align*}
M_\sigma \colonequals (u, f_{\sigma(1)}) \oplus  (u, f_{\sigma(1)} f_{\sigma(2)}) \oplus \dots \oplus  (u, f_{\sigma(1)} \dots  f_{\sigma(n-1)}). 
\end{align*}
\end{definition}
\begin{thm}[\cite{[IW3]}, 5.1]
For a $cA_{m-1}$ singularity $\rR \cong\mathbb{C}\llbracket u,x,x,v \rrbracket/ (uv-f_1 \dots f_{n})$, the maximal rigid objects in $\ucmrr$ are precisely the objects $\{ M_\sigma \mid \sigma \in S_n\}$. In particular, if $\rR$ is isolated then there are $n!$ maximal rigid objects, and each has $n-1$ summands.
\end{thm}

\begin{center}
\begin{figure}[h!]
 \captionsetup{width=0.9\linewidth}

\begin{tikzpicture}[scale=1.3,bend angle=15, looseness=1,>=stealth]
\node (C1) at (60:1.5cm) [] {$ \scriptstyle M_\id$};
\node (C2) at (120:1.5cm) [] {$ \scriptstyle M_{(23)}$};
\node (C3) at (180:1.5cm) [] {$ \scriptstyle M_{(132)}$};
\node (C4) at (240:1.5cm) [] {$ \scriptstyle M_{(13)}$};
\node (C5) at (300:1.5cm) [] {$\scriptstyle M_{(123)}$};
\node (C6) at (0:1.5cm) [] {$\scriptstyle M_{(12)}$};
\draw[->, bend right]  (C1) to (C6);
\draw[->, bend right]  (C6) to (C1);
\draw[->, bend right]  (C1) to (C2);
\draw[->, bend right]  (C2) to (C1);
\draw[->, bend right]  (C2) to (C3);
\draw[->, bend right]  (C3) to (C2);
\draw[->, bend right]  (C3) to (C4);
\draw[->, bend right]  (C4) to  (C3);
\draw[->, bend right]  (C5) to (C4);
\draw[->, bend right]  (C4) to  (C5);
\draw[<-, bend right]  (C5) to  (C6);
\draw[<-, bend right]  (C6) to  (C5);
\node at (30:1cm) {$\scriptstyle s_1$};
\node at (30:1.6cm) {$\scriptstyle s_1$};
\node at (90:1cm) {$\scriptstyle s_2$};
\node at (90:1.6cm) {$\scriptstyle s_2$};
\node at (150:1cm) {$\scriptstyle s_1$};
\node at (150:1.6cm) {$\scriptstyle s_1$};
\node at (210:1cm) {$\scriptstyle s_2$};
\node at (210:1.6cm) {$\scriptstyle s_2$};
\node at (270:1cm) {$\scriptstyle s_1$};
\node at (270:1.6cm) {$\scriptstyle s_1$};
\node at (330:1cm) {$\scriptstyle s_2$};
\node at (330:1.6cm) {$\scriptstyle s_2$};
\end{tikzpicture}
\hspace{1cm}
\begin{tikzpicture}[scale=1.3,bend angle=15, looseness=1,>=stealth]
\node (C1) at (60:1.5cm) [] {$ \scriptstyle A_\con$};
\node (C2) at (120:1.5cm) [] {$ \scriptstyle B_\con$};
\node (C3) at (180:1.5cm) [] {$ \scriptstyle C_\con$};
\node (C4) at (240:1.5cm) [] {$ \scriptstyle A_\con$};
\node (C5) at (300:1.5cm) [] {$\scriptstyle B_\con$};
\node (C6) at (0:1.5cm) [] {$\scriptstyle C_\con$};
\draw[->, bend right]  (C1) to (C6);
\draw[->, bend right]  (C6) to (C1);
\draw[->, bend right]  (C1) to (C2);
\draw[->, bend right]  (C2) to (C1);
\draw[->, bend right]  (C2) to (C3);
\draw[->, bend right]  (C3) to (C2);
\draw[->, bend right]  (C3) to (C4);
\draw[->, bend right]  (C4) to  (C3);
\draw[->, bend right]  (C5) to (C4);
\draw[->, bend right]  (C4) to  (C5);
\draw[<-, bend right]  (C5) to  (C6);
\draw[<-, bend right]  (C6) to  (C5);
\node at (30:1cm) {$\scriptstyle s_1$};
\node at (30:1.6cm) {$\scriptstyle s_1$};
\node at (90:1cm) {$\scriptstyle s_2$};
\node at (90:1.6cm) {$\scriptstyle s_2$};
\node at (150:1cm) {$\scriptstyle s_1$};
\node at (150:1.6cm) {$\scriptstyle s_1$};
\node at (210:1cm) {$\scriptstyle s_2$};
\node at (210:1.6cm) {$\scriptstyle s_2$};
\node at (270:1cm) {$\scriptstyle s_1$};
\node at (270:1.6cm) {$\scriptstyle s_1$};
\node at (330:1cm) {$\scriptstyle s_2$};
\node at (330:1.6cm) {$\scriptstyle s_2$};
\end{tikzpicture}
\caption{The left hand side shows $\umrig(\ucmrr)$ for an isolated $cA_{m-1}$ singularity given by $uv-f_1f_2f_3$. The right hand side shows the `picture' of the derived equivalence class of the contraction algebras where paths determine tilting complexes. } \label{graphs}
\end{figure}
\end{center}

In particular, if we choose $n=3$ , there are six maximal rigid objects, each with two summands and the mutation graph is shown in Figure \ref{graphs}. Note that in this geometric setting, it is shown in \cite[\S1, \S 7]{[IW4]} that fixing the ordering of the summands of one rigid object in $\ucmrr$ fixes an ordering on all the other rigid objects in $\mut(M)$, independent of the mutation sequences taken. Thus, each $N \in \mut(M)$ will appear precisely once as a vertex of $\umut(M)$.
\begin{figure}[h!] 
\begin{center}
 \captionsetup{width=0.9\linewidth}
\begin{tikzpicture} [bend angle=45,looseness=1]
\node[vertex] (a) at (0,0) {};
\node[vertex] (b) at (1.8,0) {};
\node at (-3.4,0) {$A_\con \colonequals \underline{\mathrm{End}}_R(M_{\id}) \cong \uEnd_R(M_{(13)}) \cong$};
\draw[<-,bend left] (b) to node[below] {$\scriptstyle c$} (a);
\draw[<-,bend left] (a) to node[above] {$\scriptstyle a$} (b);
\node (C5) at (3.4, 0.1) {$\scriptstyle cl=0 $};
\node (C6) at (3.4, -0.3) {$\scriptstyle la=0 $};
\node (C7) at (3.4, 0.5) {$\scriptstyle l^2+acacac=0 $};
\draw [->] (b) edge [in=40,out=-40,loop,looseness=6] node[below,pos=0.3] {$\scriptstyle l$} (b);
\end{tikzpicture}
\vspace{0.2cm}
\begin{tikzpicture} [bend angle=45,looseness=1]
\node[vertex] (a) at (0,0)  {};
\node[vertex] (b) at (1.8,0) {};
\node at (-4,0) {$B_\con \colonequals \underline{\mathrm{End}}_R(M_{(23)})\cong \uEnd_R(M_{(123)}) \cong$};
\draw[<-,bend left] (b) to node[below] {$\scriptstyle c$} (a);
\draw[<-,bend left] (a) to node[above] {$\scriptstyle a$} (b);
\node (C5) at (3.3, 0.1) {$\scriptstyle la=0 $};
\node (C6) at (3.3, -0.3) {$\scriptstyle cl=0 $};
\node (C7) at (3.3, 0.5) {$\scriptstyle l^2=ac $};
\node (C8) at (4.3, 0.1) {$\scriptstyle am=0 $};
\node (C9) at (4.3, -0.3) {$\scriptstyle mc=0 $};
\node (C10) at (4.3, 0.5) {$\scriptstyle m^3=ca $};
\draw [->] (b) edge [in=40,out=-40,loop,looseness=6] node[below,pos=0.3] {$\scriptstyle l$} (b);
\draw [->] (a) edge [in=140,out=-140,loop,looseness=6] node[below,pos=0.3] {$\scriptstyle m$} (a);
\end{tikzpicture}
\vspace{0.2cm}
\begin{tikzpicture} [bend angle=45,looseness=1]
\node[vertex] (a) at (0,0)  {};
\node[vertex] (b) at (1.8,0) {};
\node at (-4,0) {$C_\con \colonequals \underline{\mathrm{End}}_R(M_{(132)})\cong \uEnd_R(M_{(12)}) \cong$};
\draw[<-,bend left] (b) to node[below] {$\scriptstyle c$} (a);
\draw[<-,bend left] (a) to node[above] {$\scriptstyle a$} (b);
\node (C5) at (2.9, 0.1) {$\scriptstyle cm=0 $};
\node (C6) at (2.9, -0.3) {$\scriptstyle ma=0 $};
\node (C7) at (2.9, 0.5) {$\scriptstyle m^3+acac=0 $};
\draw [->] (a) edge [in=140,out=-140,loop,looseness=6] node[below,pos=0.3] {$\scriptstyle m$} (a);

\end{tikzpicture}
\vspace{-0.2cm}
\caption{The quivers and relations of the contraction algebras of the minimal models of the $cA_2$ singularity given by $uv-xy(x^2+y^3)$.} \label{quivers}
\end{center}
\end{figure}

Going back to the $n=3$ example, by Theorem \ref{HMMPbij} there are six minimal models of any such $\Spec \rR$, each with two curves in the exceptional locus. Choosing $f_1=x$, $f_2=y$ and $f_3=x^2+y^3$ the contraction algebras associated to these minimal models are given in Figure \ref{quivers}.

By Corollary \ref{geommaxrigid}, these are the only basic members of a derived equivalence class. Further, as in Example \ref{ex2}, we obtain a picture of this derived equivalence class, shown on the right hand side of Figure \ref{graphs}, which controls the tilting complexes and hence the derived equivalences.

\begin{remark}
More generally, if $A_{\con}$ is the contraction algebra of a minimal model of an isolated $cA_{m-1}$ singularity given by $uv-f_1\dots f_n$, then the quiver of $A_{\con}$ will be the double of the $A_{n-1}$ Dynkin quiver, possibly with up to two loops at each vertex \cite[5.29]{[IW3]}.   
\end{remark}
 
\subsection{Global Setting}

In this final subsection, we remove the restrictions in Setup \ref{comploc} that the base of the flopping contraction needs to be complete local, or even affine.
\begin{setup} \label{global}
Take $f \colon X \to X_\con$ to be a $3$-fold flopping contraction between quasi-projective varieties where $X$ has at worst Gorenstein terminal singularities.
\end{setup}

In this more general setup, Donovan--Wemyss introduce a more general invariant given by a sheaf of algebras \cite{Enhance}. As with the construction of the contraction algebra, the construction involves a vector bundle $\mathcal{V} \colonequals \mathcal{O}_X \oplus \mathcal{V}_0$ on $X$ satisfying
\begin{align*}
f_* \EuScript{E}nd_X(\mathcal{V}) \cong  \cE nd_{X_\con}(f_*\mathcal{V}).
\end{align*}
Although this bundle may not be tilting (as it is in the complete local case) there is a technical condition on $\mathcal{V}$, detailed in \cite[2.3]{Enhance}, which ensures that for any choice of affine open $\Spec R$ in $X_\con$, the bundle $\mathcal{V}|_{f^{-1}(\Spec R)}$ is a tilting bundle.

With this bundle $\mathcal{V}$, they define the \textit{sheaf of contraction algebras} to be
\begin{align*}
\mathcal{D} \colonequals f_* \EuScript{E}nd_X(\mathcal{V})/\mathcal{I}
\end{align*}
where $\mathcal{I}$ is the ideal sheaf of local sections that at each stalk at $v \in X_\con$ factor through a finitely generated projective $\mathcal{O}_{X_{\con},v}$-module (see \cite[2.8]{Enhance} for details).

Writing $Z$ for the locus of points on $X_\con$ above which $f$ is not an isomorphism, \cite[2.16]{Enhance} showed that 
the support of the sheaf $\mathcal{D}$ is precisely $Z$. In particular, in the setup of \ref{global}, the condition on $X$ ensures that $Z= \{p_1, \dots, p_n\}$ where each $p_i$ is an isolated singularity and thus
\begin{align*}
\mathcal{D} \cong \bigoplus_{i=1}^n \mathcal{D}_{p_i}
\end{align*}
where $\mathcal{D}_{p_i}$ is the $\mathcal{O}_{X_{\con},p_i}$-algebra given by the stalk of $\mathcal{D}$ at $p_i$. Specifically, in the setup of \ref{global}, $\mathcal{D}$ is a finite dimensional algebra which splits into a direct sum of algebras, one for each point $p_i$.

Alternatively, for each $p_i$, it is possible to choose an affine neighbourhood $R_i$ of $p_i$ which contains no other $p_j$. Localising if necessary, we can assume $p_i$ is the unique closed point of $R_i$ and setting $U_i \colonequals f^{-1}(\Spec R_i)$, we can consider the map $f_i \colonequals f|_{U_i}$. Further, we can complete this map to obtain a map
\begin{align*}
\widehat{f}_i \colon \widehat{U}_i \to \Spec \widehat{R}_i.
\end{align*} 
This map now satisfies the conditions of the complete local setup in \ref{comploc} and thus we get an associated contraction algebra $A_i \colonequals \uEnd_{\widehat{R}_i}(N_i)$ where $N_i$ is a rigid object in $\underline{\mathrm{CM}}\widehat{R}_i$. 
\begin{thm}\cite[2.24]{Enhance} \label{di}
The completion of $\mathcal{D}_{p_i}$ is Morita equivalent to $A_i$.
\end{thm}
As $\mathcal{D}_{p_i}$ is a finite length module over $\mathcal{O}_{X_{\con},p_i}$, there is an isomorphism $\widehat{\mathcal{D}}_{p_i} \cong \mathcal{D}_{p_i}$ of $\mathcal{O}_{X_{\con},p_i}$-algebras, where $\widehat{\mathcal{D}}_{p_i}$ denotes the completion of $\mathcal{D}_{p_i}$. Combining this with our earlier results gives the following.
\begin{thm}
Under the setup of \ref{global} every algebra derived equivalent to $\mathcal{D}$ is Morita equivalent to an algebra of the form 
\begin{align*}
\bigoplus_{i=1}^n \uEnd_{\widehat{R}_i}(M_i)
\end{align*}
where $M_i \in \mut_{\widehat{R}_i}(N_i)$. In particular, there are only finitely many basic algebras in the derived equivalence class.
\end{thm}
\begin{proof}
Any algebra derived equivalent to $\mathcal{D}$ must be of the form
\begin{align*}
\bigoplus_{i=1}^n B_i
\end{align*}
where $B_i$ is derived equivalent to $\mathcal{D}_{p_i}$. However, by Theorem \ref{di} and the remark after, $\mathcal{D}_{p_i}$ is Morita equivalent to $A_i \colonequals \uEnd_{\widehat{R}_i}(N_i)$ and thus each $B_i$ must be derived equivalent to $A_i$. However, for each $i$, Corollary \ref{mainresult} shows that the only basic algebras in the derived equivalence class of $A_i$ are $\uEnd_{\widehat{R}_i}(M_i)$ for $M_i \in \mut_{\widehat{R}}(N_i)$ and thus $B_i$ must be Morita equivalent to one of these algebras.
\end{proof}

\end{document}